\documentclass{article}
\usepackage[utf8]{inputenc}

\usepackage[normalem]{ulem}

\title{Serre's uniformity question and proper subgroups of $C_{ns}^+(p)$}
\author{Lorenzo Furio and Davide Lombardo}
\date{}

\usepackage[usenames,dvipsnames]{color}
\usepackage{amsmath}%
\usepackage{amsfonts}%
\usepackage{amssymb}%
\usepackage{amsthm}%
\usepackage{graphicx}
\usepackage{multirow}
\usepackage[utf8]{inputenc}
\usepackage[english]{babel}
\usepackage{mathrsfs}
\usepackage{array}
\usepackage{rotating}
\usepackage{multirow}
\usepackage{faktor}
\usepackage{tikz}
\usepackage{bbm}

\usepackage{a4wide}

\usepackage{hyperref}

\hypersetup{
	colorlinks,
	citecolor=black,
	linkcolor=black,
	urlcolor=black
}

\newtheorem{theorem}{Theorem}[section]

\newtheorem{claim}[theorem]{Claim}

\newtheorem{corollary}[theorem]{Corollary}

\newtheorem{lemma}[theorem]{Lemma}

\newtheorem{proposition}[theorem]{Proposition}

\newtheorem{question}[theorem]{Question}

\theoremstyle{definition}
\newtheorem{definition}[theorem]{Definition}

\newcommand{\cubes}{\operatorname{cubes}}

\newcommand{\boundonprimes}{103000}

\theoremstyle{remark}
\newtheorem{remark}[theorem]{Remark}
\newcommand{\Z}{\mathbb{Z}}
\newcommand{\Q}{\mathbb{Q}}
\newcommand{\C}{\mathbb{C}}
\newcommand{\F}{\mathbb{F}}
\newcommand{\OK}{\mathcal{O}_K}

\newcommand{\Gal}{\operatorname{Gal}}

\newcommand{\Aut}{\operatorname{Aut}}
\newcommand{\GL}{\operatorname{GL}}

\newcommand{\Fheight}{\operatorname{h}_\mathcal{F}}
\newcommand{\uhp}{\mathcal{H}}

\numberwithin{equation}{section}

\begin{document}
	
\maketitle

\begin{abstract}
    Serre's uniformity question asks whether there exists a bound $N>0$ such that, for every non-CM elliptic curve $E$ over $\Q$ and every prime $p>N$, the residual Galois representation $\rho_{E,p}:\Gal(\overline{\Q}/\Q) \to \Aut(E[p])$ is surjective. 
    The work of many authors has shown that, for $p>37$, this representation is either surjective or has image contained in the normaliser of a non-split Cartan subgroup $C_{ns}^+(p)$. Zywina has further proved that, whenever $\rho_{E,p}$ is not surjective for $p>37$, its image is either $C_{ns}^+(p)$ or a certain subgroup $G(p)$ of $C_{ns}^+(p)$ of index 3. Recently, Le Fourn and Lemos showed that the index-3 case cannot arise for $p>1.4 \cdot 10^7$. We strengthen this result by proving that the image of $\rho_{E, p}$ is not conjugate to $G(p)$ for any prime larger than $5$.
\end{abstract}

\tableofcontents

\section{Introduction}

In 1972 Serre \cite{serre72} proved his celebrated Open Image Theorem, stating that for a number field $K$ and an elliptic curve $\faktor{E}{K}$ without (potential) complex multiplication, the $p$-adic representation \[\rho_{E,p^\infty} : \Gal\left(\faktor{\overline{K}}{K}\right) \to \Aut(T_pE)\] is surjective for almost every prime $p$, where $T_pE$ is the $p$-adic Tate module of $E$. He also gave a stronger formulation of this result: the image of the adelic representation \[
\rho_E = \prod_{p} \rho_{E, p^\infty} : \Gal\left(\faktor{\overline{K}}{K}\right) \to \prod_{p \text{ prime}} \Aut(T_pE) \cong \Aut(\widehat{\Z}^2)
\]
is open; equivalently, the image of $\rho_E$ has finite index in $\Aut(\widehat{\Z}^2)$.
In the same paper, Serre proposed what has become known as Serre's uniformity question, which has proved to be extremely hard even for elliptic curves defined over $\Q$.
\begin{question}[Serre's uniformity question]\label{question: uniformity for Q, mod-p version}
    Let $K$ be a number field. Does there exist a constant $N$, depending only on $K$, such that for every non-CM elliptic curve $\faktor{E}{K}$ and for every prime $p>N$ the residual representation $$\rho_{E,p} : \Gal\left(\faktor{\overline{K}}{K}\right) \to \Aut(E[p]) \cong \operatorname{GL}_2(\F_p)$$ is surjective?
\end{question}A lemma of Serre \cite[IV-23, Lemma 3]{MR1484415} shows that, for $p \geq 5$ unramified in $K$, the $p$-adic representation attached to $\faktor{E}{K}$ is surjective if and only if the mod-$p$ representation is. Thus, Question \ref{question: uniformity for Q, mod-p version} is essentially equivalent to the problem of the surjectivity of the $p$-adic representations attached to elliptic curves $\faktor{E}{K}$, for all $p$ large enough.
In the rest of this paper, we focus on the case $K=\mathbb{Q}$, for which it is widely believed that Question \ref{question: uniformity for Q, mod-p version} has an affirmative answer with $N=37$. 

To show that the image $\operatorname{Im}\rho_{E,p}$ of $\rho_{E, p}$ is all of $\GL_2(\F_p)$, one can try to exclude the possibility that $\operatorname{Im}\rho_{E,p}$ is contained in a maximal proper subgroup of $\GL_2(\F_p)$. Taking into account some basic constraints that $\operatorname{Im}\rho_{E,p}$ has to satisfy -- namely, that its determinant is all of $\F_p^\times$ and that it should contain an element of order 2 with eigenvalues $\pm 1$ (the image of complex conjugation) -- there remains quite a short list of maximal subgroups of $\GL_2(\F_p)$ that could potentially contain $\operatorname{Im}\rho_{E,p}$: some so-called `exceptional' subgroups, the Borel subgroups, and the normalisers of (split or non-split) Cartan subgroups. Progress on Question \ref{question: uniformity for Q, mod-p version} for $K=\mathbb{Q}$ has mostly been organised according to this classification of maximal subgroups. 

Serre himself showed \cite[§8.4, Lemma 18]{MR644559} that the exceptional subgroups cannot contain the image of $\rho_{E, p}$ as soon as $p>13$.
Mazur \cite{mazur78} then proved that there are no isogenies of prime degree $p$ between non-CM elliptic curves over $\mathbb{Q}$ for $p>37$: this is equivalent to the fact that for $p>37$ the image of $\rho_{E, p}$ is not contained in a Borel subgroup. More recently, Bilu and Parent developed their version of Runge's method for modular curves \cite{bilu11runge}, which allowed them to prove \cite{bilu11split} that $\operatorname{Im}\rho_{E,p}$ is not contained in the normaliser of a split Cartan for sufficiently large $p$. The result was then sharpened by Bilu-Parent-Rebolledo \cite{bilu13}, who showed that the same statement holds for every $p \geq 11$, with the possible exception of $p=13$.
Finally, the result was extended to also cover the prime $p=13$ by means of the so-called \textit{quadratic Chabauty} method \cite{split-level13}.

Thus, the only case that remains open is that of normalisers of non-split Cartan subgroups. Runge's method cannot be applied to the modular curves corresponding to these subgroups, but partial results have been obtained even in this case. In particular, Zywina \cite{zywina15} and Le Fourn--Lemos \cite{lefournlemos21} studied the possibility that the image of $\rho_{E, p}$ is \textit{strictly} contained in the normaliser of a non-split Cartan subgroup. We now recall their results.

Assuming $p>2$, we let $\varepsilon$ be the reduction modulo $p$ of the least positive integer which represents a quadratic non-residue in $\F_p^\times$. We denote by $C_{ns}(p)$ the non-split Cartan group
\begin{equation}\label{eq:cartan}
    C_{ns}(p):= \left\lbrace \begin{pmatrix} a & \varepsilon b \\ b & a \end{pmatrix} \,\middle|\, a,b \in \F_p, \ (a,b) \ne (0,0) \right\rbrace
\end{equation}
and by $C_{ns}^+(p)$ its normaliser.
\begin{theorem}[Zywina]\label{zywina}
    Suppose that $\rho_{E,p}$ is not surjective for a non-CM elliptic curve $\faktor{E}{\Q}$ and a prime $p>37$. \begin{itemize}
        \item If $p \equiv 1 \pmod 3$, then $\operatorname{Im}\rho_{E,p}$ is conjugate to $C_{ns}^+(p)$ in $\operatorname{GL}_2(\F_p)$.
        \item If $p \equiv 2 \pmod 3$, then $\operatorname{Im}\rho_{E,p}$ is conjugate in $\operatorname{GL}_2(\F_p)$ either to $C_{ns}^+(p)$ or to the group $$G(p):= \{a^3 \bigm\vert a \in C_{ns}(p)\} \cup \left\lbrace \begin{pmatrix} 1 & 0 \\ 0 & -1 \end{pmatrix} \cdot a^3 \bigm\vert a \in C_{ns}(p) \right\rbrace \subset C_{ns}^+(p).$$
    \end{itemize}
\end{theorem}
This theorem follows from Proposition 1.13 of the unpublished preprint \cite{zywina15}. A full proof 
is available in print as \cite[Proposition 1.4]{lefournlemos21}.

Note that the definition of $G(p)$ can be extended to all $p>3$. For $p=3$, the order of $C_{ns}^+(3)$ is prime to $3$, hence $\{a^3 \bigm\vert a \in C_{ns}(3)\} = C_{ns}(3)$ and $G(3)=C_{ns}^+(3)$. For $p=2$, the definition of $C_{ns}^+(2)$ is slightly different from \eqref{eq:cartan}, and we do not give it here; however, we remark that the possible images of $\rho_{E, 2}$ and $\rho_{E, 3}$ are well-understood, see \cite[Theorems 1.1 and 1.2]{zywina15}.

In their paper \cite{lefournlemos21}, Le Fourn and Lemos studied the case where $\operatorname{Im}\rho_{E,p}$ is conjugate to $G(p)$, ruling out this possibility for all sufficiently large primes:
\begin{theorem}[Le Fourn, Lemos {\cite[Theorem 1.2]{lefournlemos21}}]\label{lefournlemos}\label{intejer}
    Let $\faktor{E}{\Q}$ be an elliptic curve without complex multiplication. If $p>37$ is a prime number such that $\operatorname{Im} \rho_{E,p} \cong G(p)$, then $p < 1.4 \cdot 10^7$ and $j(E) \in \Z$.
\end{theorem}
In the introduction to \cite{{lefournlemos21}}, the authors describe the difficulties in extending their result to primes smaller than $1.4 \cdot 10^7$.
In particular, they show that for any prime $p>37$ and elliptic curve $\faktor{E}{\Q}$ without CM for which $\operatorname{Im}\rho_{E,p} \cong G(p)$, we have $\log|j(E)| \le \max\{12000, 7\sqrt{p}\} \le 27000$, which together with the fact that $j(E) \in \Z$ shows that there are only finitely many ($\overline{\mathbb{Q}}$-isomorphism classes of) curves to check. However, as they point out, there seems to be no easy way to handle the remaining cases algorithmically. 
For comparison, Bajolet, Bilu, and Matschke \cite{bilu21} stated the following claim, whose proof relied in part on a computation that took several CPU years.
\begin{claim}[Bajolet, Bilu, Matschke]\label{bilucartan}
Let $p \ge 11$ be a prime and let $P \in X_{ns}^+(p)(\Q)$ be a non-CM point such that $j(P) \in \Z$. We have $p>100$.
\end{claim}
Unfortunately, it seems that the proof of this claim contains a gap (see the PhD thesis of the first author \cite[Remark 5.3.10]{furio2025phd} for details). Even so, the scale of the computation needed to analyse the primes up to $100$ illustrates how challenging it would be to extend Theorem \ref{lefournlemos} to all primes up to $1.4 \cdot 10^7$ by brute-force calculation alone.
The aim of this article is to deal with the remaining primes in Theorem \ref{lefournlemos} and to prove the following simple dichotomy in Serre's uniformity question. \begin{theorem}\label{solution}
	Let $\faktor{E}{\Q}$ be an elliptic curve without complex multiplication and let $p>37$ be a prime number.
	The image of $\rho_{E,p}$ is either $\GL(E[p])$ or the normaliser of a non-split Cartan subgroup of $\GL(E[p])$.
\end{theorem}

For $p \leq 37$, the images of the mod-$p$ representations attached to elliptic curves over $\Q$ have been studied extensively (see \cite{zywina15, MR4468989, balakrishnan23} for the state of the art). Combined with Theorem \ref{solution}, this allows us to show the following theorem.
\begin{theorem}\label{solution for small primes}
    The only prime $p \ne 2,3$ for which there exists a non-CM elliptic curve $\faktor{E}{\Q}$ such that $\operatorname{Im}\rho_{E,p}$ is conjugate to $G(p)$ is $p=5$.
\end{theorem}

Similar to the observation of Le Fourn and Lemos in \cite[Theorem 1.3]{lefournlemos21}, Theorems \ref{solution} and \ref{solution for small primes} also completely settle a question of Najman \cite{najman18}, improving \cite[Theorem 1.3]{lefournlemos21}. Let $d \ge 1$ be a positive integer and let $I_\Q(d)$ be the set of prime numbers $p$ for which there exists a rational elliptic curve $\faktor{E}{\Q}$ without complex multiplication and an isogeny $\varphi : E \to E'$ of degree $p$ defined over a field $K$ of degree $[K:\Q] \le d$. From Mazur's work \cite{mazur78} we know that $I_\Q(1) = \{2,3,5,7,11,13,17,37\}$, and Najman's question concerns the sets $I_{\mathbb{Q}}(d)$ for $d \geq 2$. As a consequence of \cite[Proposition 1.4]{lefournlemos21} and Theorem \ref{solution for small primes} we obtain:
\begin{theorem}
    For every positive integer $d$ we have
    $$I_\Q(d) = I_\Q(1) \cup \{p \text{ prime} \mid p \le d-1\}.$$
\end{theorem}
This is an unconditional version of \cite[Theorem 4.1]{najman18}, and the proof relies on the same arguments.

\medskip

We now discuss the ideas behind the proof of Theorem \ref{solution}. While the general approach is similar to that of Le Fourn and Lemos, based on Runge's method for modular curves, many of the details differ, and we improve the estimates involved in many ways, both quantitatively and qualitatively. 

We begin by reformulating Theorem \ref{solution} in terms of rational points on modular curves. 
Let $X_{G(p)}$ be the compactified modular curve defined over $\Q$ with function field $\Q(X(p))^{G(p)}$ (under the canonical isomorphism $\Gal\left(\faktor{\Q(X(p))}{\Q(X(1))}\right) \cong \operatorname{GL}_2(\F_p)$). The set of complex points $X_{G(p)}(\C)$ can be identified with the quotient of $\uhp^* = \uhp \cup \Q \cup \{\infty\}$ by the action of the group $\Gamma = \{\gamma \in \operatorname{SL}_2(\Z) \mid \gamma \pmod p \in G(p)\}$. More generally, given a field $K \subseteq \C$, the set $X_{G(p)}(K) \setminus \{\text{cusps}\}$ can be identified with the set of $\overline{K}$-isomorphism classes of elliptic curves $\faktor{E}{K}$ such that the image of the Galois representation $\rho_{E,p}\left(\Gal\left(\faktor{\overline{K}}{K}\right)\right)$ is contained in $G(p)$ up to conjugacy.
Using this interpretation for the rational points of $X_{G(p)}$, it is not difficult to see that Theorem \ref{solution} is equivalent to the following statement.
\begin{theorem}\label{thm: reformulation}
    Let $p>37$ be a prime number. Every rational point $P \in X_{G(p)}(\Q)$ is either a cusp or a CM point.
\end{theorem}

The proof of Theorem \ref{lefournlemos} by Le Fourn and Lemos is based on two fundamental steps: first, they show that an elliptic curve satisfying the hypothesis of Theorem \ref{lefournlemos} has integral $j$-invariant (via the formal immersion method of Mazur). Second, they prove an upper bound on $|j(E)|$ by combining Runge's method
with an effective surjectivity theorem showing that $\operatorname{Im} \rho_{E, p} = \GL(E[p])$ for all $p$ greater than an explicit bound depending on $j(E)$. 

The first step works in complete generality: Theorem \ref{lefournlemos} gives the integrality of $j(E)$ as soon as $p>37$, so -- in order to prove Theorem \ref{solution} -- we can assume $j(E) \in \mathbb{Z}$. Our main contribution lies in a much sharper upper bound on $|j(E)|$, which we achieve through three main innovations:
\begin{itemize}
    \item We prove a sharp effective surjectivity theorem (in the spirit of \cite{MR1209248}, \cite{MR3437765}, and \cite[Theorem 5.2]{lefourn16}) by refining the proof of the effective isogeny theorem of Gaudron and R\'emond \cite{gaudron-remond}. We obtain substantially improved constants by showing that certain auxiliary subvarieties considered in the proof are all trivial in our case (see Lemma \ref{B[sigma]=0}).
    \item Second, we perform a detailed analysis of the local properties of the representations $\rho_{E, p}$. This analysis yields several improvements, such as ruling out all primes $p \equiv -1 \pmod{9}$ and proving that $p^4$ divides $j(E)$. Furthermore, we show that $j(E)$ can be written as $p^k c^3$ for some integer $c$. When we eventually reduce the proof of Theorem \ref{solution} to an explicit calculation, this latter relation has the effect of dividing by three \textit{on a logarithmic scale} the number of tests we have to perform, significantly reducing the computational component of our approach.
    
\item Finally, the third and most significant innovation is our much more detailed study of the modular units on the curve $X_{G(p)}$. The main ingredients that lead to our improved bound on $\log |j(E)|$ are sharp bounds on character sums, which essentially draw on Weil's method to treat Kloosterman sums \cite{weil48}, an idea based on Abel's summation to amplify certain cancellation phenomena among roots of unity, and direct computations to fully exploit the extent of these cancellations.
All of these improvements are crucial to lowering the bound on $\log|j(E)|$ to values that are computationally tractable, and the result we obtain is sharp enough that the final computation takes less than two minutes of CPU time.
\end{itemize}

We now elaborate more on some of these ingredients.
The first component of our approach is Theorem \ref{effiso}, which gives an upper bound (in terms of the Faltings height of $E$) on the largest prime for which the image of the representation $\rho_{E,p}$ is contained in the normaliser of a non-split Cartan subgroup. The main improvement with respect to the existing literature consists in the fact that our version applies to all elliptic curves, even those with very small height. This is essential for our application, because the estimate $\log |j(E)| \leq 27000$ obtained in \cite{lefournlemos21} already shows that the Faltings height of $E$ is not too large (and our improved bounds later show that in fact $\log |j(E)| < 162$, leading one to consider curves with very small height). We point out that, while Theorem \ref{effiso} is expressed in terms of the Faltings height of $E$, this can be related very precisely to the Weil height of $j(E)$ (hence to $\log |j(E)|$) thanks to Theorem \ref{heights-ineq}.

Next, we describe the results we obtain by studying the behaviour of the restriction of $\rho_{E,p}$ to inertia subgroups. By considering the image of $p$-inertia, we prove in Theorem \ref{-1mod9} that, for $p \equiv -1 \pmod 9$, the image of the representation $\rho_{E,p}$ is not contained in the group $G(p)$ up to conjugacy. Since we already know from Theorem \ref{zywina} that it suffices to consider primes $p \equiv 2 \pmod 3$, this result immediately rules out one third of the remaining primes. We also show that this statement cannot be improved by purely local arguments (see Remark \ref{rmk: local arguments cannot be extended further}). This reduction in the number of primes to consider turns out to be quite useful, though not essential, in a later step of the proof.
Studying the image of $\ell$-inertia for $\ell \ne p$ we further prove that, if $\operatorname{Im} \rho_{E, p}$ is conjugate to $G(p)$, then the $\ell$-adic valuation of the $j$-invariant is divisible by 3, hence $j(E)$ is almost a cube (that is, it is an integer of the form $p^k c^3$ with $c \in \mathbb{Z}$). Moreover, using the theory of the canonical subgroup of Lubin \cite{lubin79} and Katz \cite{katz73}, we show that $j(E)$ is divisible by $p^4$ (Proposition \ref{canonicalsgr}). This gives the inequality $\log p \le \frac{1}{4} (\log|j(E)| - 3\log |c|)$, which for small values of $j(E)$ is better than the bound on $p$ coming from our surjectivity theorem, and in all cases further constrains the possible values of $j(E)$ for a given $p$. All of these results are new with respect to \cite{lefournlemos21}; in particular, our use of the canonical subgroup to study the arithmetic properties of $j$ seems to be novel.

Finally, we review the central point of the argument, which leads to the sharp upper bound $\log |j(E)| < 40$. While we follow a strategy that is broadly similar to the proof of Theorem \ref{lefournlemos} by Le Fourn and Lemos, we improve their results by several new means. A key part of the argument, based on Runge's method, concerns the size of the values of a suitable modular unit $U$ when evaluated at integral points of $X_{G(p)}$: since we know that $U(P)$ is integral and divides $p^3$ (where $P$ is the point on the modular curve corresponding to our elliptic curve $E$), this can be leveraged to obtain an upper bound on $\log |j(E)|$. The Fourier expansion of the specific modular unit that we employ involves certain arithmetic quantities related to the distribution of cubes in the finite field $\F_{p^2}$, encoded by complicated sums of roots of unity.
We employ ideas from Weil's strategy for the study of Kloosterman sums \cite{weil48} to take advantage of cancellations among these roots of unity, proving Proposition \ref{logq<p}, which gives a bound of the form $\log |j(E)| \ll p^{1/4}$, where the approach of \cite{lefournlemos21} only yields $\log |j(E)| \ll p^{1/2}$ (in both cases, the implied constants can be made explicit). 
Combined with the surjectivity theorem, which gives the competing bound $p \ll \log |j(E)|$, this result allows us to reduce the primes under consideration to those smaller than $1.03 \cdot 10^5$ and to show that $\log|j(E)| < 162$. This estimate is already hugely better than the bound $\log|j(E)| \le 27000$ obtained in \cite{lefournlemos21}.
We point out that we also correct some minor arithmetic errors in \cite{lefournlemos21}, which we review in Section \ref{sec:modular} and which would lead to a slightly worse bound in Theorem \ref{lefournlemos}. However, this fact does not affect our proof, as we rework the arguments completely.

Unfortunately, the bound $\log|j(E)| < 162$ turns out to be still too large to be of practical use in a direct computation. Using Abel's summation we then rewrite the sums of roots of unity we already estimated in a different form, which we expect to amplify as much as possible the cancellation phenomena. We are not able to prove cancellation analytically (at least, not to the extent that we need), but luckily, the quantities under consideration are sums of roots of unity depending only on $p$, and in particular, they are independent of the hypothetical elliptic curve $E$ whose mod-$p$ representation has image conjugate to $G(p)$. We can then determine their exact values by a direct numerical calculation, which -- given our previous results -- we need only perform for primes $p <1.03 \cdot 10^5$ with $p \equiv 2, 5 \pmod{9}$. In Remark \ref{rmk: heuristics for the bound s<p} we give a heuristic explanation of why one should expect this calculation to yield bounds of the quality we find in practice.
Having already reduced the number of primes to a reasonable amount is crucial at this stage, since the cost of computing the relevant sums of roots of unity for a given prime $p$ takes time proportional to $p \log p$, and it would be impossible to carry out this calculation for all primes up to the original bound of $1.4 \cdot 10^7$. Re-inserting this new numerical bound in our study of modular units leads to Proposition \ref{logq<30}, which gives $\log|j(E)| < 40$. Note that here we depart significantly from the strategy of Bilu-Parent and Le Fourn-Lemos: our absolute bound on $|j|$ comes from a detailed study of modular units, while the bound on $p$ serves mainly as a tool to simplify certain analytic estimates (which in principle could also be proved by different means).

We are now left with a significantly smaller number of $j$-invariants to consider. Since the condition $\operatorname{Im}\rho_{E,p} \cong G(p)$ is invariant under twists, we can study each $j$-invariant separately by writing down any elliptic curve $\faktor{E}{\mathbb{Q}}$ with that $j$-invariant. As we also have an absolute bound on $p$, this reduces the proof of Theorem \ref{solution} to a finite computation, which would however still require considerable effort. However, from our local analysis, we know that $j(E)$ is of the form $p^k \cdot c^3$ with $k \in \{4,5\}$ ($k=6$ cannot occur, see Lemma \ref{p|j}), so the bound on $\log |j(E)|$ gives a corresponding bound on $\log |c|$, which is now very tractable (by way of example, for $p=2003$ and $k=4$ our bounds show that $|c| \leq 31$, thus leaving only 62 $j$-invariants to test -- $j=0$ corresponds to a CM elliptic curve).
We conclude the proof by checking directly that none of the remaining pairs $(E, p)$ satisfies $\operatorname{Im} \rho_{E, p} \cong G(p)$.

We conclude this introduction by briefly describing the contents of each section. In Section \ref{sec:preliminaries} we discuss some preliminary results, including an asymptotically optimal comparison between the Faltings height of rational elliptic curves and their modular height (Theorem \ref{heights-ineq}). 
In Section \ref{sec:local} we study the local properties of the representations $\rho_{E, p}$, proving all the results mentioned above. 
In Section \ref{sec:isogeny} we show our effective surjectivity theorem (Theorem \ref{effiso}). Section \ref{sec:modular} is the core of the article: it contains in particular our analysis of modular units, leading to the improved bounds discussed above. 
Finally, in Section \ref{sec:Conclusion} we describe the details of the short computation that concludes the proofs of Theorems \ref{solution} and \ref{solution for small primes}.

\paragraph{Acknowledgements.} We are grateful to Samuel Le Fourn for many helpful discussions and his interest in this work and to Pascal Autissier for an interesting conversation about partial sums of arithmetic functions. We would also like to thank \'Eric Gaudron, Filip Najman, Ga\"el Rémond, and Andrew Sutherland for their comments on an earlier version of this article. D.L. was supported by the University of Pisa through grant PRA-2022-10 and by MUR grant PRIN-2022HPSNCR (funded by the European Union project Next Generation EU) and is a member of the INdAM group GNSAGA.

\section{Preliminaries}\label{sec:preliminaries}

In this section we collect some auxiliary results needed for the proof of Theorem \ref{solution}. In particular, we provide explicit comparisons between the stable Faltings height and the modular height of an elliptic curve over the rationals.

We start by recalling the following lemma, which can be found in \cite[Lemma 3.5]{bilu13}.

\begin{lemma}[Bilu, Parent, Rebolledo]\label{logsum}
	For every $x \in (0,1)$ we have
	$$-\sum_{k=1}^{\infty} \log(1-x^k) < - \frac{\pi^2}{6\log x}.$$
\end{lemma}

The next two lemmas will be useful for the parts of our argument that rely on the complex interpretation of modular curves.
\begin{lemma}\label{lemminotecnico1}
	Let $p$ be a positive integer. Let $\tau \in \uhp$ be a point in the standard fundamental domain for the action of $\operatorname{SL}_2(\Z)$, let $q=e^{2\pi i \tau}$, and let $q^{\frac{1}{p}}$ be the $p$-th root of $q$ given by $(e^{2\pi i \tau})^{\frac{1}{p}} = e^\frac{2\pi i \tau}{p}$. We have
    $$|1-q^\frac{1}{p}| \le 1-|q|^\frac{1}{p} + |q|^\frac{1}{2p} \frac{\pi}{p}.$$
\end{lemma}
\begin{proof}
	Since $\tau$ is in the standard fundamental domain, we have $|\Re\{\tau\}| \le \frac{1}{2}$, hence we can write $q^\frac{1}{p}=|q|^\frac{1}{p} e^{i\theta}$ with $|\theta| \le \frac{\pi}{p}$. By using that $\sqrt{a^2+b^2} \le |a|+|b|$ for all real numbers $a$ and $b$, we obtain
	\begin{align*}
		|1-q^\frac{1}{p}| &= \sqrt{(1-|q|^\frac{1}{p}\cos \theta)^2 + |q|^\frac{2}{p} \sin^2 \theta} \\
		&= \sqrt{1- 2|q|^\frac{1}{p}\cos \theta + |q|^\frac{2}{p}} \\
		&= \sqrt{(1-|q|^\frac{1}{p})^2 + 2|q|^\frac{1}{p}(1- \cos\theta)} \\
		&\le 1-|q|^\frac{1}{p} + |q|^\frac{1}{2p}|\theta|,
	\end{align*}
	and the lemma follows. 
\end{proof}

\begin{lemma}\label{lemminotecnico2}
	Let $p>1$ be an integer and let $x \in (0,1)$. We have
	\begin{enumerate}
		\item $1-x^\frac{1}{p} < \frac{|\log x|}{p}$;
		\item $\frac{x^\frac{1}{p}}{1-x^\frac{1}{p}} < \frac{p}{|\log x|}$.
	\end{enumerate}
\end{lemma}
\begin{proof}
	Both results are obtained from the inequality $\log y <y-1$, with $y=x^\frac{1}{p}$ and $y=x^{-\frac{1}{p}}$ respectively.
\end{proof}

The following lemma is not elementary, since for its proof we need to rely on both \cite[Theorem 1.2]{lefournlemos21} and the determination of the rational points on the curve $X_{ns}^+(17)$ \cite{balakrishnan23}.

\begin{lemma}\label{lemma:theREALintejer}
	Let $\faktor{E}{\Q}$ be an elliptic curve without complex multiplication. If $p>5$ is a prime number such that $\operatorname{Im} \rho_{E,p} \cong G(p)$, where $G(p)$ is the group defined in Theorem \ref{zywina}, then $p \ge 19$ and $j(E) \in \Z$.
\end{lemma}
\begin{proof}
We first show $p \ge 19$ and $p \neq 37$. Theorem \ref{zywina} gives $p \equiv 2 \pmod 3$, and by assumption we have $p>5$, so it only remains to rule out the primes $11$ and $17$. The case $p=11$ cannot occur by \cite[Theorem 1.6(i)]{zywina15}, while the case $p=17$ cannot occur by \cite[Theorem 1.2]{balakrishnan23}. Now suppose by contradiction that $j(E)$ is not an integer: \cite[Theorem 1.2]{lefournlemos21} then gives $p \in \{7,11,13,17,37\}$, which contradicts what we just showed.
\end{proof}

\subsection{Faltings height and modular height of rational elliptic curves}

We now give an upper bound on the stable Faltings height of an elliptic curve over $\Q$ in terms of its $j$-invariant. Any elliptic curve $\faktor{E}{\Q}$ can also be considered as an elliptic curve  over $\C$, so there exists a complex number $\tau \in \uhp$ such that $E(\C) \cong \faktor{\C}{\Z \oplus \tau\Z}$. We fix such a $\tau$ and set $q=e^{2\pi i \tau}$. Our results in this section refine the properties of heights explained in \cite{silverman86}.

\begin{definition}
    We will consider the \emph{standard fundamental domain for the action of $\operatorname{SL}_2(\Z)$} as $$\mathcal{F} := \left\lbrace z \in \uhp : \Re\{z\} \in \left( -\frac{1}{2}, \frac{1}{2}\right], |z| > 1 \right\rbrace \cup \left\lbrace e^{i\theta} : \frac{\pi}{3} \le \theta \le \frac{\pi}{2} \right\rbrace.$$
\end{definition}

We begin with the following theorem, which combines \cite[Corollary 2.2]{bilu11runge} with \cite[Lemma 2.5]{pazuki19}.
\begin{theorem}\label{estimate-qj}
	Let $\tau \in \uhp$ be in the standard fundamental domain $\mathcal{F}$ and let $\faktor{E}{\C}$ be the corresponding elliptic curve. Set $q=e^{2\pi i \tau}$. We have $$\log|j(E)| \le \max\{\log3500, |\log|q||+\log2 \}$$ 
	and $$|\log|q|| \le \log(|j(E)|+970.8) < \log|j(E)| + \frac{970.8}{|j(E)|}.$$
\end{theorem}
\begin{proof}
	The first inequality follows from \cite[Corollary 2.2]{bilu11runge}, while the second one is obtained from \cite[Lemma 2.5]{pazuki19} using the fact that $\log(x+a)-\log x = \log\left(1+\frac{a}{x}\right) < \frac{a}{x}$.
\end{proof}

\begin{corollary}\label{estimate-qj-cor}
	In the setting of Theorem \ref{estimate-qj}, if $|j(E)| \ge 3500$, then
	$$\log|j(E)| - \log2 \le |\log|q|| \le \log|j(E)| + 0.245.$$
\end{corollary}
\begin{proof}
	We only need to notice that $|\log|q||<\log|j(E)|+\log\left(1+\frac{970.8}{3500}\right)<\log|j(E)|+0.245$.
\end{proof}
Before stating the precise comparisons between heights that we need, we record the following fact that we will use often.

\begin{theorem}\label{qisreal}
	Let $\faktor{E}{\mathbb{R}}$ be an elliptic curve isomorphic to $\faktor{\C}{\Z \oplus \tau\Z}$ and let $q=e^{2\pi i \tau}$. If $\tau$ is in the standard fundamental domain $\mathcal{F}$ for the action of $\operatorname{SL}_2(\Z)$, then either $q \in \mathbb{R}$, or $j(E) \in (0,1728)$ (equivalently, $|\tau|=1$).
\end{theorem}
\begin{proof}
    By \cite[Proposition V.2.1]{ataoec} we know that the $j$-function gives a bijection between $\mathbb{R}$ and the set $\mathcal{C} = C_1 \cup C_2 \cup C_3$, where $C_1 = \{it \mid t\ge 1\}$, $C_2 = \left\lbrace e^{i\theta} \mid \frac{\pi}{3} \le \theta \le \frac{\pi}{2} \right\rbrace$ and $C_3 = \left\lbrace \frac{1}{2} + it \mid t \ge \frac{\sqrt{3}}{2} \right\rbrace$. Moreover, by continuity, it is easy to notice that $j(C_1)=[1728,+\infty)$, $j(C_2)=[0,1728]$ and $j(C_3)=(-\infty,0]$. Hence, if $j(E) \not\in (0,1728)$, then $\Re \tau \in \left\lbrace 0, \frac{1}{2} \right\rbrace$, which concludes the proof.
\end{proof}

In the next result, as well as in the rest of the paper, we denote by $\Fheight(E)$ the stable Faltings height of an elliptic curve $E$, with the normalisation of \cite[Section 1.2]{deligne85}.

\begin{theorem}\label{heights-ineq}
    Let $\faktor{E}{\Q}$ be an elliptic curve with stable Faltings height $\Fheight(E)$. Let $\tau \in \uhp$ be the point in the standard fundamental domain $\mathcal{F}$ such that $E(\C) \cong \faktor{\C}{\Z \oplus \tau\Z}$, and set $q=e^{2\pi i \tau}$.
	\begin{enumerate}
		\item If $|j(E)| > 3500$, then 
  \begin{equation}\label{eq: heights-ineq-1}
  \frac{\operatorname{h}(j(E))}{12} - \frac{1}{2}\log\log|j(E)| -0.406 < \Fheight(E) < \frac{\operatorname{h}(j(E))}{12} - \frac{1}{2}\log\log|j(E)| +0.159,
  \end{equation}
		where $\operatorname{h}(x)$ is the logarithmic Weil height of $x$ (i.e., if $x=\frac{a}{b}$ with $(a,b)=1$, we set $\operatorname{h}(x)=\log\max\{|a|,|b|\}$).
		\item If $j \in \Z$, then
\begin{equation}\label{eq: heights-ineq-2}		
  \Fheight(E) < -\frac{1}{12}\log|q| - \frac{1}{2}\log|\log|q|| -\frac{1}{2}\log 2 - \frac{\pi^2}{3\log|q|}.
  \end{equation}
		\item We always have
\begin{equation}\label{eq: heights-ineq-3}		
  \Fheight(E) > -\frac{1}{12}\log|q| - \frac{1}{2}\log|\log|q|| -\frac{1}{2}\log2 -\frac{2|q|}{1-|q|}.
	\end{equation}
 \end{enumerate}
\end{theorem}

\begin{remark}
	It is well known that $|\Fheight(E)- \frac{\operatorname{h}(j)}{12}| < \log(\operatorname{h}(j) + 1) + O(1)$ (see \cite{silverman86}), but the above theorem also gives a bound in the opposite direction, namely, $|\Fheight(E)- \frac{\operatorname{h}(j)}{12}| > \frac{1}{2}\log\log|j| + O(1)$.
\end{remark}

\begin{proof}
	By \cite[Proposition 1.1]{silverman86} we have
\begin{equation}\label{eq: Silverman's formula for the Faltings height}
 \Fheight(E)= \frac{1}{12[K:\Q]} \left( \log|N_{\faktor{K}{\Q}}(\Delta_{\faktor{E}{K}})| - \sum_{v \in M_K^\infty} n_v \log\left(|\Delta(\tau_v)|(\pi^{-1}\Im\{\tau_v\})^6\right) \right),
 \end{equation}
	where $\faktor{K}{\Q}$ is a finite Galois extension over which $E$ has semistable reduction everywhere, $\Delta_{\faktor{E}{K}}$ is the minimal discriminant of $E$ over $K$, $M_K^\infty$ is the set of the Archimedean places of $K$, $\tau_v$ is an element of $\uhp$ such that $E(\overline{K}_v) \cong \faktor{\C}{\Z + \tau_v \Z}$, and $\Delta(\tau)=(2\pi)^{12} q \prod_{n=1}^\infty (1-q^n)^{24}$, with $q=e^{2\pi i \tau}$. Note that \cite[Proposition 1.1]{silverman86} is formulated with a different normalisation of the Faltings height, but it is easy to convert from Silverman's convention to Deligne's:  specifically, the height $h$ in \cite[Proposition 1.1]{silverman86} satisfies $h(E) = \Fheight(E) - \frac{1}{2} \log \pi$. This difference is reflected in the factor $\pi^{-1}$ in equation \eqref{eq: Silverman's formula for the Faltings height}.

	Let $e_p,f_p$ be respectively the ramification index and inertia degree in $K$ of the rational prime $p$, and let $r_p$ be the number of distinct primes of $\mathcal{O}_K$ dividing $p$ (the numbers $e_p,f_p$ only depend on $p$ since $\faktor{K}{\Q}$ is a Galois extension). Given that $j=j(E)$ is a rational number, we have
	\begin{align*}
		N_{\faktor{K}{\Q}}(\Delta_{\faktor{E}{K}}) &= \prod_{\substack{Q \subset \OK \\ \text{prime}}} \left| \frac{\OK}{Q^{\max\{0,-v_Q(j)\}}} \right| = \prod_{p \text{ prime}}\left(p^{f_p \max\{0,-e_p v_p(j)\}}\right)^{r_p} \\
		&= \prod_{p \text{ prime}}\left( p^{\max\{0,-v_p(j)\}} \right)^{[K:\Q]} = \prod_{p \text{ prime}}\left( \max\{1,\|j\|_p\} \right)^{[K:\Q]},
	\end{align*}
	where the first equality holds by \cite[Table 15.1]{aoec} and the fact that $E$ has semistable, hence in particular multiplicative, reduction at primes dividing the discriminant.
	For every $v \in M_K^\infty$ we can assume that $\tau_v$ belongs to the standard fundamental domain $\mathcal{F}$, so, since $E$ is defined over $\Q$, we may use the same $\tau_v=\tau$ for every $v$. Writing $\Im\{\tau\} = -\frac{\log|q|}{2\pi} = \frac{|\log |q||}{2 \pi}$, we have
	\begin{equation}\label{eqheight}
	    12\Fheight(E) = \sum_{p \text{ prime}}\log\max\{1,\|j\|_p\} -\log|q| -6\log2 -6\log|\log|q|| -24\sum_{n=1}^\infty \log|1-q^n|.
	\end{equation}
	Using the fact that, for every $z \in \C$ such that $|z|<1$, we have $|\log|1-z|| \le -\log|1-|z||$, from Lemma \ref{logsum} we obtain
	\begin{align*}
		-24\sum_{n=1}^\infty \log|1-q^n| < - 24\sum_{n=1}^{\infty} \log(1-|q|^n) < -\frac{4\pi^2}{\log|q|}.
	\end{align*}
	Replacing in equation \eqref{eqheight}, we get
	$$\Fheight(E) < \frac{1}{12} \left( \sum_{p \text{ prime}}\log\max\{1,\|j\|_p\} -\log|q| -6\log2 -\frac{4\pi^2}{\log|q|} -6\log|\log|q|| \right). $$
	We note that for $j \in \Z$ we have $\|j\|_p \le 1$ for every prime $p$, and \eqref{eq: heights-ineq-2} follows.
 
	To prove the upper bound in \eqref{eq: heights-ineq-1}, we note that $\log|j| = \log\max\{1,|j|\}$, and using Corollary \ref{estimate-qj-cor} together with the assumption $|j(E)| \geq 3500$ we obtain
	\begin{align*}
		\Fheight(E) &< \frac{1}{12} \left( \sum_{p \text{ prime}}\log\max\{1,\|j\|_p\} +\log\max\{1,|j|\} +0.245 \right. \\
		& \qquad \qquad \left. - 6\log2 - \frac{4\pi^2}{\log|q|} - 6\log|\log|q|| \right) \\
		&= \frac{\operatorname{h}(j)}{12} +\frac{0.245}{12} -\frac{1}{2}\log2 -\frac{\pi^2}{3\log|q|} -\frac{1}{2}\log|\log|q|| \\
		&< \frac{\operatorname{h}(j)}{12} -0.326 +\frac{\pi^2}{3(\log|j|-\log2)} -\frac{1}{2}\log(\log|j|-\log2) \\
		&< \frac{\operatorname{h}(j)}{12} -\frac{1}{2}\log\log|j| +0.159.
	\end{align*}
	On the other hand, using that $\log(1+x) < x$ for every $x>0$, we have
	$$-2\sum_{n=1}^\infty \log|1-q^n| \ge -2\sum_{n=1}^\infty \log(1+|q|^n) > -2\sum_{n=1}^\infty |q|^n = -\frac{2|q|}{1-|q|}.$$
Noting that $\log\max\{1,\|j\|_p\} \ge 0$ for every $p$, we see from equation \eqref{eqheight} that the inequality in \eqref{eq: heights-ineq-3} holds. To prove the lower bound in \eqref{eq: heights-ineq-1}, we use $\log|j| > -\log|q|>\log|j|-\log2$ (Theorem \ref{estimate-qj}) in equation \eqref{eqheight} to obtain
	\begin{align*}
		\Fheight(E) &> \frac{1}{12}\sum_{p \text{ prime}}\log\max\{1,\|j\|_p\} + \frac{1}{12}\log|j| - \frac{7}{12}\log2 -\frac{1}{2}\log\log|j| - \frac{4}{|j(E)|-2} \\
		&> \frac{\operatorname{h}(j)}{12} - \frac{1}{2}\log\log|j| - \frac{7}{12}\log2 - \frac{2}{1749},
	\end{align*}
	which concludes the proof.
\end{proof}

\begin{remark}
	The above argument even gives $$\Fheight(E) \le \frac{\operatorname{h}(j)}{12} - \frac{1}{2}\log\log|j| -\frac{1}{2}\log2 + o(1) \; \text{ as $|j| \to \infty$,}$$ yielding a better constant term as $j$ grows. Explicitly, one has
	$$\Fheight(E) < \frac{\operatorname{h}(j)}{12} -\frac{1}{2}\log\log|j| -\frac{1}{2}\log2 + \left(\frac{\log2}{2}+\frac{\pi^2}{3}\right)\frac{1}{\log|j| - \log2} + \frac{970.8}{|j|},$$
	where we have used $\log(x-\log2)-\log x = - \sum_n \frac{(\log 2)^n}{nx^n} > -\sum_n \frac{(\log2)^n}{x^n} = -\frac{\log2}{x-\log2}$, with $x=\log|j|$.
\end{remark}

\begin{remark}\label{minimalheight}
    Minimising the function $$-\frac{1}{12}\log x - \frac{1}{2}\log|\log x| - \frac{1}{2}\log2 - \frac{2x}{1-x}$$ over the interval $(0,e^{-\pi\sqrt{3}}]$, we obtain that for every elliptic curve $\faktor{E}{\Q}$ we have $\Fheight(E)>-0.74885$. This fits well with the computation by Deligne of the absolute minimum of the height \cite[pag. 29]{deligne85}. With our normalisation, Deligne has shown that the minimum of $\Fheight(E)$ is approximately $-0.74875$, attained for the elliptic curve with $j=0$, for which $|q|=e^{-\pi\sqrt{3}}$. Moreover, this is the minimum height for elliptic curves over any number field.
\end{remark}
\begin{remark}
	For $j \in \Z$, Theorem \ref{heights-ineq} implies that $$\Fheight(E) = -\frac{1}{12}\log|q| -\frac{1}{2}\log|\log|q|| -\frac{1}{2}\log2 + O\left(\frac{1}{\log|q|}\right)$$
	as $|q| \to 0$.
\end{remark}

\section{Local analysis}\label{sec:local}

Let $\faktor{E}{\Q}$ be an elliptic curve without CM such that, for some prime $p$, we have $\operatorname{Im}\rho_{E,p} \subseteq G(p)$ up to conjugacy.
In this section, we consider $E$ as an elliptic curve over $\mathbb{Q}_\ell$ for various primes $\ell$ (including $\ell=p$), and study the representation $\rho_{E, p}$ upon restriction to $\Gal\left(\faktor{\overline{\Q_\ell}}{\Q_\ell}\right)$, considered as a decomposition subgroup of $\Gal\left(\faktor{\overline{\Q}}{\Q}\right)$.
This allows us to obtain the congruence $p \equiv 2, 5 \pmod{9}$ (Theorem \ref{-1mod9}) and to prove various arithmetic properties of $j(E)$: in particular, we show that $j(E)$ is of the form $p^k c^3$ for some integers $c$ and $k \geq 0$  (Lemma \ref{lem:jisalmostacube}) and that $p^4 \mid j(E)$ (Proposition \ref{canonicalsgr}). 

\subsection{The image of the inertia subgroups}

We start by considering the image via $\rho_{E,p}$ of the $\ell$-inertia subgroup, for all primes $\ell$. We will draw different conclusions in the cases $\ell \ne p$ and $\ell = p$. We first consider the case $\ell = p$, which allows us to show that, for $p \equiv -1 \pmod9$, the image of the residual representation modulo $p$ cannot be isomorphic to the subgroup $G(p)$, and hence $\operatorname{Im} \rho_{E,p} \supseteq C_{ns}^+(p)$ up to conjugacy. This is a refinement of \cite[Proposition 1.4]{lefournlemos21} (Theorem \ref{zywina}), itself an exposition of results of Zywina \cite[Proposition 1.13]{zywina15}, which states that if $\operatorname{Im}\rho_{E,p} \subseteq G(p)$, then $p \equiv 2 \pmod 3$. With respect to Theorem \ref{zywina}, the congruence $p \not \equiv -1 \pmod{9}$ allows us to exclude approximately one-third of the primes from our subsequent analysis,
since we are left with only two of the three classes $2, 5, -1$ modulo 9.
\begin{theorem}\label{-1mod9}
	Let $\faktor{E}{\Q}$ be an elliptic curve without complex multiplication. If $p \ne 2,3$ is a prime number such that $\operatorname{Im} \rho_{E,p}$ is conjugate to $G(p)$, then $p \equiv 2 \pmod 3$ and $p \not\equiv -1 \pmod9$.
\end{theorem}
	
\begin{proof}
    Since $\det \circ \rho_{E,p}$ surjects onto $\F_p^\times$ and $\det(G(p)) = (\F_p^\times)^3$, we have $(\F_p^\times)^3 = \F_p^\times$ and hence $p \equiv 2 \pmod 3$. Since the smallest prime $p$ congruent to $-1$ modulo 9 is 17, and since we already noticed that $p$ must be equal to 2 modulo 3, it suffices to consider primes $p \ge 17$.
	Choosing a suitable basis of $E[p]$, we can suppose $\operatorname{Im} \rho_{E,p} = G(p)$. We note that $\Gal\left(\faktor{\overline{\Q}_p}{\Q_p}\right)$ can be identified with a $p$-decomposition group of $\Gal\left(\faktor{\overline{\Q}}{\Q}\right)$, and $I_p:=\Gal\left(\faktor{\overline{\Q}_p}{\Q_p^{nr}}\right)$ can be identified with the $p$-inertia, where $\Q_p^{nr}$ is the maximal unramified extension of $\Q_p$. We also let $\Q_p^{tame}$ be the maximal tamely ramified extension of $\Q_p^{nr}$. We have:	
	\begin{align*}
		\operatorname{Im} \rho_{E,p} &= \rho_{E,p} \left(\Gal\left(\faktor{\overline{\Q}}{\Q}\right)\right) > \rho_{E,p} \left(\Gal\left(\faktor{\overline{\Q}_p}{\Q_p}\right)\right) > \rho_{E,p} \left(\Gal\left(\faktor{\overline{\Q}_p}{\Q_p^{nr}}\right)\right) \\
		&= \rho_{E,p}(I_p) =: I.
	\end{align*}
	 Notice that the restriction of $\rho_{E,p}$ to $\Gal\left(\faktor{\overline{\Q}_p}{\Q_p^{tame}}\right)$ is trivial, because its image is a $p$-group contained in $G(p)$, which has $\frac{2(p^2-1)}{3}$ elements. Hence $\rho_{E,p}$ factors through the quotient, inducing a map from $\Gal\left(\faktor{\Q_p^{tame}}{\Q_p^{nr}}\right)$ which we still denote by $\rho_{E,p}$. We have that
	\begin{align*}
		I &= \rho_{E,p}(I_p) \cong \rho_{E,p} \left(\Gal\left(\faktor{\Q_p^{tame}}{\Q_p^{nr}}\right)\right) \cong \Gal\left(\faktor{\Q_p^{nr}(E[p])}{\Q_p^{nr}}\right).
	\end{align*}	

    As shown in \cite[Appendix B]{lefournlemos21}, $v_p(j(E)) \ge 0$ for $p>5$ (i.e., for every $p$ such that $p-1 \nmid 2(p+1)$), hence $E$ has potentially good reduction at $p$. Let $\faktor{K}{\Q_p^{nr}}$ be the minimal extension of $\Q_p^{nr}$ over which $E \times_{\operatorname{Spec} \Q} \operatorname{Spec} \Q_p^{nr}$ acquires good reduction. Consider the subgroup $I_K < I_p$ given by $I_K := \Gal\left(\faktor{\overline{\Q}_p}{K}\right)$.
	By \cite[Section 5.6]{serre72} we know that $e:=[K:\Q_p^{nr}] \in \{1,2,3,4,6\}$. By \cite[Section 1, Propositions 10 and 11]{serre72} we know that either $\rho_{E,p}(I_K)$ is cyclic of order $\frac{p^2-1}{e}$, or the image of $\rho_{E,p}(I_K)$ in $\operatorname{PGL}_2(\F_p)$ contains an element of order $\frac{p-1}{\gcd(e, p-1)}$.
	In the latter case, since the square of any element of $C_{ns}^+(p) \setminus C_{ns}(p)$ is a scalar matrix and hence has order $2$ in $\operatorname{PGL}_2(\F_p)$, every element in $C_{ns}^+(p)$ has order dividing $p+1$ in $\operatorname{PGL}_2(\F_p)$. We then have that $\frac{p-1}{(e, p-1)} \mid p+1$, and so $p-1 \mid 2e \le 12$. However, this is impossible for $p \ge 17$.
	This 
    shows that $\left|\Gal\left(\faktor{K(E[p])}{K}\right)\right| = [K(E[p]):K]= \frac{p^2-1}{e}$.
	\begin{center}
		\begin{tikzpicture}[auto]
			\node (Qpnr) at (0,0) {$\Q_p^{nr}$};
			\node (K) at (1.5,1.5) {$K$};
			\node (KEp) at (0,3) {$K(E[p])$};
			\node (QpnrEp) at (-1.5,1.5) {$\Q_p^{nr}(E[p])$};
			
			\draw[-] (Qpnr) to node [swap] {$e$} (K);
			\draw[-] (Qpnr) to node {$|I|$} (QpnrEp);
			\draw[-] (K) to node [swap]	{$\frac{p^2-1}{e}$} (KEp);
			\draw[-] (QpnrEp) to node {} (KEp);
		\end{tikzpicture}
	\end{center}
	Since $\Gal\left(\faktor{\Q_p^{tame}}{\Q_p^{nr}}\right)$ is a procyclic group, $\Gal\left(\faktor{K(E[p])}{\Q_p^{nr}}\right)$ is cyclic (of order $\frac{p^2-1}{e} \cdot e = p^2-1$). This implies that $I$ is contained in $C_{ns}(p)$, as every element of $C_{ns}^+(p) \setminus C_{ns}(p)$ has order dividing $2(p-1)$, but a generator of $I$ has order at least $\frac{p^2-1}{6}$ and $\frac{p^2-1}{6} > 2(p-1)$ for $p>11$. As noticed in \cite[Appendix B]{lefournlemos21}, $\frac{p^2-1}{|I|}$ is necessarily odd, for otherwise $I$ would be contained in the subgroup of squares of $C_{ns}(p)$, contradicting the fact that $\det \circ \rho_{E,p}|_{I_p}$ surjects onto $\F_p^\times$. Moreover, $\frac{p^2-1}{|I|}$ divides $e$, hence it is either 1 or 3. However, if we had $|I|=p^2-1$, then the whole $I=C_{ns}(p)$ would be contained in $G(p)$, which yields a contradiction. Hence we must have $|I|=\frac{p^2-1}{3}$ and $e \in \{3,6\}$.
	Given that $p \equiv 2 \pmod3$, we have $e \mid \frac{p^2-1}{3} \ \Longleftrightarrow \ 3e \mid p^2-1 \ \Longleftrightarrow \ p \equiv -1 \pmod9$. In particular, whenever $p \equiv -1 \pmod9$, we have that $e$ divides $|I|$, and so $\Q_p^{nr}(E[p])$ has a subextension of degree $e$. Since $\faktor{K(E[p])}{\Q_p^{nr}}$ is cyclic, it has a unique subextension of degree $e$, hence $K \subset \Q_p^{nr}(E[p])$, and so $\Q_p^{nr}(E[p]) = K(E[p])$. This implies that $|I|= p^2-1$, and so $\operatorname{Im} \rho_{E,p}$ cannot be contained in $G(p)$.
\end{proof}

\begin{remark}\label{rmk: local arguments cannot be extended further}
    Theorem \ref{-1mod9} provides the best congruence condition on $p$ that can be obtained by local arguments at $p$. Indeed, as shown in \cite[Proposition 1.16 (iv)]{zywina15}, if $p \equiv 2,5 \pmod 9$, the CM elliptic curve $E: \ y^2=x^3+16p^k$, with $k \equiv - \frac{p+1}{3} \pmod 3$, is such that $\operatorname{Im}\rho_{E,p}$ is conjugate to $G(p)$. However, there exist elliptic curves $E'$ without CM whose defining equations are arbitrarily close to that of $E$ in the $p$-adic metric. By continuity of the local $p$-adic representation with respect to the coefficients of a defining equation (Krasner's lemma), taking $E'$ sufficiently close to $E$ gives examples of non-CM elliptic curves for which the image of $\rho_{E',p}$, restricted to the decomposition group at $p$, is contained in $G(p)$.
\end{remark}

We now consider the action of the $\ell$-inertia for $\ell \ne p$, which allows us to prove the following lemma.
\begin{lemma}\label{lem:jisalmostacube}
    Let $\faktor{E}{\Q}$ be an elliptic curve without complex multiplication. If $p \geq 19$ is a prime number such that $\operatorname{Im} \rho_{E,p}$ is conjugate to $G(p)$, then $j(E)=p^d \cdot c^3$, with $d,c \in \Z$ and $d \ge 0$.
\end{lemma}
\begin{proof}
    By Lemma \ref{lemma:theREALintejer} we know that $j(E) \in \Z$. Let $\ell \ne p$ be a prime that divides $j(E)$, let $\Q_\ell^{nr}$ be the maximal unramified extension of $\Q_\ell$, and let $\faktor{K}{\Q_\ell^{nr}}$ be the minimal extension over which $E_\ell = E \times_{\operatorname{Spec}\Q} \operatorname{Spec}\Q_\ell^{nr}$ acquires good reduction. Let $y^2=x^3+ax+b$ be a minimal model for $E$ over $\Q$ with discriminant $\Delta$. By the N\'eron-Ogg-Shafarevich criterion, we know that $K$ is the minimal extension of $\Q_\ell^{nr}$ such that $\rho_{E_\ell,p}\left(\Gal\left(\faktor{\overline{\Q}_\ell}{K}\right)\right)$ is trivial, hence $\Gal\left(\faktor{\overline{\Q}_\ell}{K}\right) = \ker \rho_{E_\ell,p}$ and $\operatorname{Im}\rho_{E_\ell,p} \cong \Gal\left(\faktor{K}{\Q_\ell^{nr}}\right)$. In particular, $\Gal\left(\faktor{K}{\Q_\ell^{nr}}\right)$ embeds in $G(p)$. Since by Theorem \ref{-1mod9} we know that $3 \nmid |G(p)|$, this implies that $3 \nmid [K:\Q_\ell^{nr}]$. However, by \cite[Proposition 1 and Theorems 1, 2, 3]{kraus90} we know that if $3 \nmid [K:\Q_\ell^{nr}]$, then $3 \mid v_\ell(\Delta)$, and hence $v_\ell(j(E)) = v_p\left(-12^3 \cdot \frac{(4a)^3}{\Delta}\right) = 3v_\ell(12) + 3v_\ell(4a) - v_\ell(\Delta)$ is divisible by 3.
\end{proof}

\subsection{The canonical subgroup}

This subsection is dedicated to proving the following proposition.
\begin{proposition}\label{canonicalsgr}
	Let $\faktor{E}{\Q}$ be an elliptic curve without complex multiplication. If $p \geq 19$ is a prime number such that $\operatorname{Im} \rho_{E,p}$ is conjugate to $G(p)$, then $p^4 \mid j(E)$.
\end{proposition}
The proof relies on the theory of the canonical subgroup of $E[p]$, first defined by Lubin and studied by Lubin and Katz (see \cite{lubin79} and \cite{katz73}).

Let $p$ be a prime and let $K$ be a $p$-adic field. Denote by $\mathfrak{p}$ the maximal ideal of $\mathcal{O}_K$. Let $E$ be an elliptic curve defined over $K$ with good reduction at $\mathfrak{p}$. Let $\hat{E}$ be the formal group associated with $E$ and let $E_1(K)$ be the set of the points in $E(K)$ that reduce to the origin $O$ modulo $\mathfrak{p}$.
As explained in \cite[Chapter VII, Proposition 2.2]{aoec}, there is an isomorphism $\hat{E}(\mathfrak{p}) \cong E_1(K)$. In particular, if we consider the extension $L=K(E[p])$, with prime ideal $\mathfrak{P} \mid \mathfrak{p}$, we have $\hat{E}(\mathfrak{P})[p] \cong E_1(L)[p]$. Hence, when $E$ has supersingular reduction modulo $\mathfrak{p}$, there is an isomorphism between the $p$-torsion subgroup of the formal group and the $p$-torsion subgroup of the elliptic curve, i.e., $\hat{E}(\mathfrak{P})[p] \cong E[p]$.
The group $\hat{E}(\mathfrak{P})$ is by definition the set $\mathfrak{P}$ endowed with the group structure coming from the formal group $\hat{E}$. Considering the points $\hat{P}$ of $\hat{E}(\mathfrak{P})$ as elements of $\mathfrak{P}$, we can then refer to the \emph{valuation} of $\hat{P} \in \hat{E}(\mathfrak{P})$: it is simply its valuation as an element of the field $L$.
\begin{definition}
	If there exists $\lambda \in \mathbb{R}$ such that $\{\hat{P} \in \hat{E}(\mathfrak{P})[p] \mid v(\hat{P}) \ge \lambda \}$ is an order-$p$ subgroup of $\hat{E}(\mathfrak{P})[p]$, then this is called the \emph{canonical subgroup of order $p$} of $E$.
\end{definition}
\begin{remark}
	When $E$ has ordinary reduction modulo $\mathfrak{p}$, there always exists a canonical subgroup, given by $\hat{E}(\mathfrak{P})[p] = E_1[p]$, i.e., the kernel of the reduction modulo $\mathfrak{p}$.
\end{remark}
\begin{remark}
	The isomorphism $E_1(L) \cong \hat{E}(\mathfrak{P})$ is given by the map $(x,y) \mapsto -\frac{x}{y}$, hence it is compatible with the action of the Galois group $\Gal\left(\faktor{\overline{K}}{K}\right)$.
\end{remark}
The notion of canonical subgroup can be extended to the case of elliptic curves defined over number fields.
\begin{definition}
	Let $K$ be a number field and let $\mathfrak{p} \mid p$ be a prime of $K$. Let $E$ be an elliptic curve over $K$ with potentially good reduction at $\mathfrak{p}$. Let $L$ be an extension of $K_\mathfrak{p}$ such that $E$ has good reduction over $L$. Given that $E(\overline{K})[p] = E(\overline{K}_\mathfrak{p})[p] = E[p]$, we define the \emph{canonical subgroup of order $p$} of $E$ as the canonical subgroup of order $p$ of $E$ over $L$, if this exists.
\end{definition}
\begin{definition}
	If $E$ is given by the equation $y^2=f(x)$, following \cite{deuring41} we define the Hasse invariant $A$ of $E$ for a prime $p$ as the reduction modulo $p\mathcal{O}_K$ of the coefficient of $x^{p-1}$ in $f(x)^\frac{p-1}{2}$.
\end{definition}
\begin{theorem}\label{canonicalandhasse}
	Let $p \ne 2$ be a prime, let $K$ be a number field and let $\mathfrak{p} \mid p$ be a prime of $K$. Let $\faktor{E}{K}$ be an elliptic curve with potentially good reduction at $\mathfrak{p}$ and let $A$ be its Hasse invariant. The elliptic curve $E$ has a canonical subgroup of order $p$ if and only if $v(A) < \frac{p}{p+1}$.
\end{theorem}
\begin{proof}
	If $E$ has ordinary reduction, it has a canonical subgroup and $v_p(A)=0$, hence from now on we assume that $E$ is supersingular. Let $c$ be the coefficient of $x^{\frac{p^2-p}{2}}$ in the division polynomial $\psi_p(x)$. By \cite[Theorem 4.6]{smith23} we know that $E$ has a canonical subgroup of order $p$ if and only if $v_p(c) < \frac{p}{p+1}$. However, by \cite[Theorem 1]{debry14}, we know that $c \equiv A \pmod p$. In particular, whenever $v_p(A) < 1$ we have that $v_p(c)=v_p(A)$, giving the statement of the theorem. If instead $v_p(A) \ge 1$, then $c \equiv A \equiv 0 \pmod p$, and therefore also $v_p(c) \ge 1$.
\end{proof}
The theory that leads to Theorem \ref{canonicalandhasse} is due to Lubin and Katz, but we have relied on \cite{smith23} because it formulates the results in a way that is closer to what we need.

\begin{theorem}\label{cartancanonical}
	Let $\faktor{E}{\Q}$ be an elliptic curve without complex multiplication and let $p \geq 19$ be a prime. If $\operatorname{Im}\rho_{E,p}$ is contained in the normaliser of a non-split Cartan subgroup, then $E$ does not have a canonical subgroup of order $p$.
\end{theorem}
\begin{proof}
	As shown in \cite[Appendix B]{lefournlemos21}, $E$ has potentially good reduction at $p$. Let $\faktor{K}{\Q_p^{nr}}$ be the minimal extension over which $E \times_{\operatorname{Spec} \Q} \operatorname{Spec}\Q_p^{nr}$ acquires good reduction. As shown in the proof of Theorem \ref{-1mod9}, if we consider $E$ as a curve over $K$, the group $I:= \rho_{E,p}\left(\Gal\left(\faktor{\overline{K}}{K}\right)\right) < \operatorname{Im}\rho_{E,p}$ is cyclic and contains elements of order $\frac{p^2-1}{12}$.
	Let us suppose that $E$ has a canonical subgroup of order $p$, which is then stable under the action of $\Gal\left(\faktor{\overline{K}}{K}\right)$. With respect to a suitable basis of $E[p]$, all the elements of $I$ are then upper triangular, and in particular, their order divides $p(p-1)$. However, this is impossible since $\frac{p^2-1}{12} \nmid p(p-1)$ for $p > 11$.
\end{proof}
We are now almost ready to prove Proposition \ref{canonicalsgr}. The only missing ingredient is a $p$-adic property of $j(E)$ that we establish in the next lemma.
\begin{lemma}\label{p|j}
	Under the assumptions of Proposition \ref{canonicalsgr}, we have $3 \nmid v_p(j(E))$. In particular, by Lemma \ref{lemma:theREALintejer} we have $v_p(j(E)) > 0$.
\end{lemma}
\begin{proof}
	Let $y^2=x^3+ax+b$ be a minimal model for $E$ over $\Q$ with discriminant $\Delta$ and let $\faktor{K}{\Q_p^{nr}}$ be the minimal extension over which $E \times_{\operatorname{Spec} \Q} \operatorname{Spec}\Q_p^{nr}$ acquires good reduction. As shown in the proof of Theorem \ref{-1mod9}, we have $3 \mid [K:\Q_p^{nr}]$, hence by \cite[Proposition 1]{kraus90} the denominator of $\frac{v_p(\Delta)}{12}$ is divisible by $3$, and so $3 \nmid v_p(\Delta)$. Hence $v_p(j(E)) = v_p\left(-12^3 \cdot \frac{(4a)^3}{\Delta}\right) = 3v_p(a) - v_p(\Delta)$ is not divisible by $3$.
\end{proof}
\begin{proof}[Proof of Proposition \ref{canonicalsgr}]
	Let $\faktor{K}{\Q_p^{nr}}$ be the minimal extension over which $E \times_{\operatorname{Spec} \Q} \operatorname{Spec}\Q_p^{nr}$ acquires good reduction, let $y^2=x^3+ax+b$ be a model of good reduction for $E$ over $\OK$ and let $A$ and $\Delta$ be the Hasse invariant and the discriminant of this model respectively. By Theorem \ref{cartancanonical} and Theorem \ref{canonicalandhasse} we know that $v_p(A) \ge \frac{p}{p+1}$. As the ramification index of $K$ over $\Q_p$ is $e \le 6$, we have $v_p(A) \in \frac{1}{e}\Z$, and therefore $v_p(A) \ge 1$ since $p>5$. The good reduction of $E$ implies that $v_p(\Delta)=0$, and by Lemma \ref{p|j} we have $0 < v_p(j(E)) = 3v_p(a) - v_p(\Delta) = 3v_p(a)$. Using that $\Delta = -16(4a^3+27b^2)$, $v_p(a) > 0$ and $v_p(\Delta) = 0$, we also have that $v_p(b)=0$.
	We now compute the Hasse invariant. We have
	$$(x^3+ax+b)^\frac{p-1}{2} = \sum_{i+j+k=\frac{p-1}{2}} \binom{(p-1)/2}{i,j} x^{3i} \cdot a^jx^j \cdot b^k,$$
	hence in particular
	\begin{align*}
		A &= \sum_{\substack{i+j+k=\frac{p-1}{2} \\ 3i+j=p-1}} \binom{(p-1)/2}{i,j} a^j b^k = \sum_{2j+3k=\frac{p-1}{2}} \binom{(p-1)/2}{j,k} a^j b^k.
	\end{align*}
	Since by Theorem \ref{zywina} we have $p \equiv 2 \pmod 3$, the minimum value of $j$ among all the indices in the last sum is $1$, hence it is not difficult to show that $v_p(a) = v_p(A) \ge 1$. This implies $v_p(j(E)) = 3v_p(a) \ge 3$. However, by Lemma \ref{p|j}, we know that $3 \nmid v_p(j(E))$, and so $v_p(j(E)) \ge 4$.
\end{proof}
For the proof of Proposition \ref{canonicalsgr}, the fact that $\operatorname{Im}\rho_{E,p} \subseteq G(p)$ is only needed in the proof of Lemma \ref{p|j} and to assume that $p \equiv 2 \pmod 3$, so we can repeat the whole argument without this assumption and obtain the following.
\begin{corollary}
	If $\faktor{E}{\Q}$ is an elliptic curve without complex multiplication and $p \geq 19$ is a prime such that $\operatorname{Im}\rho_{E,p}$ is contained in the normaliser of a non-split Cartan subgroup, then either $v_p(j(E))=0$ or $v_p(j(E)) \ge 3$. Moreover, in the latter case we always have $p \equiv 2 \pmod 3$.
\end{corollary}
\begin{proof}
    If $p \equiv 2 \pmod 3$ the proof is exactly the same as that of Proposition \ref{canonicalsgr}, hence it suffices to show that if $p \equiv 1 \pmod 3$ then $v_p(j(E))=0$. Following the argument in the proof of Proposition \ref{canonicalsgr}, assume by contradiction that $p \equiv 1 \pmod 3$ and $v_p(j(E)) > 0$. Over a suitable extension $K$ of $\mathbb{Q}_p^{nr}$, we can write $E$ as $y^2 = x^3 + ax + b$ with $v_p(\Delta)=v_p(b)=0$ and $3v_p(a) = v_p(j(E)) > 0$. We can then write the Hasse invariant as $A = c \cdot b^{\frac{p-1}{6}} + a \cdot d(a,b)$ for some constants $c, d(a,b) \in \mathcal{O}_K$ with $v_p(c)=0$. It follows that $v_p(A)=0$, which gives a contradiction with Theorems \ref{canonicalandhasse} and \ref{cartancanonical}.
\end{proof}

\section{An effective surjectivity theorem}\label{sec:isogeny}
	
	In this section, we give a new version of the effective surjectivity theorem of Le Fourn \cite[Theorem 5.2]{lefourn16}, obtaining a bound on the size of the largest prime $p$ for which the image of the representation $\rho_{E,p}$ is contained in the normaliser of a non-split Cartan subgroup of $\GL(E[p])$. The main difference from \cite[Theorem 5.2]{lefourn16} is that our version also applies non-trivially to curves of small height.
	\begin{theorem}[Effective surjectivity]\label{effiso}
		Let $E$ be an elliptic curve without CM defined over the number field $K$. We denote by $\Fheight(E)$ the stable Faltings height of $E$ with the normalisation of \cite[Section 1.2]{deligne85}. Let $p$ be a prime such that $\rho_{E,p} \left(\Gal\left(\faktor{\overline{K}}{K}\right)\right) \subseteq C_{ns}^+(p)$ up to conjugacy.
		\begin{enumerate}
			\item If $\Fheight(E)>1$, we have
			$$p < 2908 [K:\Q] \left( \Fheight(E) + 2\log p + \frac{3}{2}\log (\Fheight(E))  + 4.725 \right).$$
			\item If $K=\Q$ and $\tau$ is the point in the standard fundamental domain $\mathcal{F}$ of $\uhp$ such that $E(\C) \cong \faktor{\C}{\Z \oplus \tau\Z}$, then
			$$p < 2908 \left( \Fheight(E) + 2\log p + \frac{3}{2}\max\{0, \log(\Im\{\tau\})\} + 2 \right).$$
			Furthermore, if we assume that $q := e^{2\pi i \tau}$ satisfies $|\log|q|| \ge 30$, then
			$$p < 2530 \left( \Fheight(E) + 2\log p + \frac{3}{2}\log(\Im\{\tau\}) + 1.94 \right).$$
		\end{enumerate}
	\end{theorem}
	
	We follow closely the approach of \cite[Theorem 5.2]{lefourn16} and \cite[Theorem 1.4]{gaudron-remond}, but we are able to obtain much-improved constants by noticing that certain auxiliary subvarieties considered in \cite{gaudron-remond} are in fact all trivial (see Lemma \ref{B[sigma]=0}). Unlike \cite{lefourn16}, for simplicity we only consider a single prime $p$ for which $\operatorname{Im}\rho_{E,p} \subseteq C_{ns}^+(p)$, instead of bounding the product of the primes for which the representation is not surjective. We refer the reader to \cite[Theorem 5.1]{furio2025effectiveboundsadelicgalois} for a more general version of this result, where several prime powers are taken into account.
	
	We begin by recalling some crucial definitions from \cite{gaudron-remond}.
	\begin{definition}\label{def: x}
		Let $A$ be a complex abelian variety, let $B \subset A$ be an abelian subvariety of codimension $t \ge 1$, and let $L$ be a polarisation on $A$. We define
		\begin{align*}
			x(B) := \left(\frac{\deg_LB}{\deg_LA}\right)^\frac{1}{t} \qquad \text{and} \qquad x := \min_{B \subsetneq A} x(B),
		\end{align*}
		where $\deg_LA$ is the top self-intersection number of the line bundle $L$ on $A$.
	\end{definition}
	Let $(A,L)$ be a polarised abelian variety defined over a number field $K$. Fix an embedding $\sigma : K \hookrightarrow \C$ and let $(A_\sigma,L_\sigma)$ be the base-change of $(A,L)$ to $\C$ via $\sigma$. We will denote by $B[\sigma]$ a proper abelian subvariety of $A_\sigma$ such that $x(B[\sigma])=x$.
	\begin{definition}
		Let $A$ be a complex abelian variety and let $L$ be a polarisation on $A$. Let $\| \cdot \|_L$ be the norm induced by $L$ on the tangent space $t_A$, and let $\Omega_A$ be the period lattice. We define
		$$\rho(A,L) := \min\{\|\omega\|_L \mid \omega \in \Omega_A \setminus \{0\}\}.$$
	\end{definition}
	\begin{remark}\label{rmk: princ pol on ell curves}
		Let $E$ be an elliptic curve defined over a number field $K$ and let $L$ be its canonical principal polarisation. As explained in \cite[Remark 3.3]{gaudron-remond}, given an embedding $\sigma : K \hookrightarrow \C$ we have $\rho(E_\sigma,L_\sigma)^{-2} = \Im\{\tau_\sigma\}$, where $\tau_\sigma$ is the element in the standard fundamental domain $\mathcal{F}$ that corresponds to $E_\sigma$ and $L_\sigma$ is the base-change of the polarisation $L$ via $\sigma$.
	\end{remark}
	\begin{definition}\label{def: delta sigma}
		Let $A$ be an abelian variety defined over a number field $K$ and let $\sigma: K \hookrightarrow \C$ be an embedding. Let $L$ be a polarisation on $A$ and let $\operatorname{d}_\sigma$ be the distance induced by $L_\sigma$ on $t_{A_\sigma}$. We define
		$$\delta_\sigma = \min\{\operatorname{d}_\sigma(\omega,t_{B[\sigma]}) \mid \omega \in \Omega_{A_\sigma} \setminus t_{B[\sigma]} \},$$
		where $B[\sigma]$ is as in Definition \ref{def: x}.
	\end{definition}

    We now begin the proof of Theorem \ref{effiso}. With notation as in that statement, let $K'$ be a quadratic extension of $K$ such that
    	\begin{equation}\label{eq:kappaprimo}
    	    \rho_{E,p}\left(\Gal\left(\faktor{\overline{K}}{K'}\right)\right) \subseteq C_{ns}(p).
    	\end{equation}
    Note that, if $\rho_{E,p}\left(\Gal\left(\faktor{\overline{K}}{K}\right)\right)$ is already contained in $C_{ns}(p)$, we choose $K'$ to be an arbitrary quadratic extension of $K$.

	
Let $\gamma$ be an element of $C_{ns}(p)$ which is not a multiple of the identity. By taking the quotient $A$ of $E \times E$ by the subgroup $\{(x,\gamma \cdot x) \bigm\vert x \in E[p]\}$, we have an isogeny $\varphi: E\times E \to A$ defined over $K'$ such that $\deg \varphi = p^2$. There exists $\psi: A \to E \times E$ such that $\psi \circ \varphi = [p]_{E \times E}$, so $\deg\psi=p^2$. As explained in the proof of \cite[Proposition 5.1]{lefourn16}, for every embedding $\sigma : K' \hookrightarrow \C$, there is a canonical norm $\| \cdot \|_\sigma$ on the tangent space of $E_\sigma$, which contains the period lattice $\Omega_{E, \sigma}$. As in \cite[Part 7.3]{gaudron-remond} and in the proof of \cite[Proposition 5.1]{lefourn16}, we choose an embedding $\sigma_0$ such that there exists a basis $(\omega_0,\tau_{\sigma_0}\omega_0)$ of $\Omega_{E,\sigma_0}$ for which $\tau_{\sigma_0}$ is as in Remark \ref{rmk: princ pol on ell curves} and
$$\| \omega_0 \|_{\sigma_0} = \max_\sigma \min_{\omega \in \Omega_{E,\sigma} \setminus \{0\}} \| \omega \|_\sigma.$$
By Remark \ref{rmk: princ pol on ell curves}, this choice of $\sigma_0$ minimizes $\Im\{\tau_{\sigma}\}$ among all $\sigma$, as in \cite[Part 7.3]{gaudron-remond}. Let $\Omega_{A,\sigma_0}$ be the period lattice of $A_{\sigma_0}$ and $\chi \in \Omega_{A,\sigma_0}$ be such that $\mathrm{d}\psi(\chi) = (\omega_0,\tau_{\sigma_0}\omega_0)$ (such a $\chi$ exists, as shown in the proof of \cite[Theorem 5.2]{lefourn16}). 

\begin{definition}
    Setting $\omega = (\omega_0, \tau_{\sigma_0}\omega_0, \chi) \in \Omega_{E \times E \times A , \sigma_0}$, we define $A_\omega$ as the minimal abelian subvariety of $(E \times E \times A)_{\sigma_0}$ containing $\omega = (\omega_0, \tau_{\sigma_0}\omega_0, \chi)$ in its tangent space.
\end{definition}
As in the proof of \cite[Proposition 5.1]{lefourn16}, one shows that $$A_\omega := \{(\psi(z),z) \mid z \in A_{\sigma_0}\} \subset (E \times E \times A)_{\sigma_0}.$$
We then see that the complex abelian variety $A_\omega$ can be defined over $K'$. When we consider $A_{\omega}$ as being defined over $K'$, we will write $(A_\omega)_\sigma$ for its base-change to $\mathbb{C}$ along a given embedding $\sigma: K' \hookrightarrow \mathbb{C}$.  The abelian variety $A_\omega$ falls within the context of \cite[Part 7.3]{gaudron-remond}. We follow \cite{gaudron-remond}, taking into account some additional features of our specific situation throughout the proof to obtain better numerical constants.
	
We choose a polarisation on $A_\omega$ as in \cite[Part 7.3]{gaudron-remond}, namely, in the following way. Set $n = \lfloor |\tau_{\sigma_0}|^2 \rfloor$, let $L_E$ be the canonical principal polarisation on $E$ and let $\pi_1,\pi_2$ be the projections from $E \times E$ on the two copies of $E$. We consider the polarisation $L'=\pi_1^*L_E^{\otimes n} \otimes \pi_2^*L_E$ on $E \times E$ and the isogeny $f$ defined as the composition $A_\omega \xrightarrow{\sim} A_{\sigma_0} \xrightarrow{\psi} (E \times E)_{\sigma_0}$, where the first isomorphism is given by the projection $(\psi(z),z) \mapsto z$. We define the polarisation $L := f^*L'$ on $A_\omega$, and as in \cite[Part 7.3]{gaudron-remond} we compute $$\deg_L A_\omega = (\deg f)\deg_{L'} E^2 = 2np^2.$$
	
\begin{lemma}\label{slopes}
    Let $\hat{\mu}_{max}(\overline{t_{A_\omega}^\vee})$ be the quantity defined in \cite[Part 6.8]{gaudron-remond}. The inequality $$\hat{\mu}_{max}(\overline{t_{A_\omega}^\vee}) \le \Fheight(E) + 2\log p + \frac{1}{2}\log \frac{n}{\pi}$$ holds.
\end{lemma}
\begin{proof}
    It suffices to combine the proof of \cite[Lemma 7.6]{gaudron-remond} with the remark at the end of the proof of \cite[Proposition 5.1]{lefourn16}, which gives $\Fheight(A_\omega) \le 2\Fheight(E) + \log p$. 
\end{proof}

The following definition collects the notations that will be needed in the rest of the proof.
\begin{definition}\label{def: notation for the proof of the effective surjectivity theorem}
    Following \cite[Parts 6.2 and 6.3]{gaudron-remond}, we set
    $\varepsilon = \frac{3\sqrt{2}-4}{2}$, $\theta = \frac{\log2}{\pi}$ and $S_\sigma := \left\lfloor \frac{\theta\varepsilon}{x\delta_\sigma^2} \right\rfloor$, where $x$ is as in Definition \ref{def: x} and $\delta_\sigma$ is as in Definition \ref{def: delta sigma}. We further define $\mathcal{V} = \{\sigma:K' \hookrightarrow \C \mid S_\sigma \ge 1\}$.
    For every $\sigma: K' \hookrightarrow \C$ choose $B[\sigma] \subset (A_\omega)_\sigma$ as in Definition \ref{def: x}. We introduce the following quantities.
    \begin{align*}
		\aleph_1 := & \ 2\max\{0,\hat{\mu}_{max}(\overline{t_{A_\omega}^\vee})\} + 5\log 2 + \frac{2}{[K':\Q]} \sum_{\sigma \in \mathcal{V}} \log \max \left\lbrace 1, \frac{1}{\rho((A_\omega)_\sigma,L_\sigma)} \right\rbrace \\
		& + \frac{4}{[K':\Q]} \sum_{\sigma \in \mathcal{V}} \log\deg_{L_\sigma}B[\sigma] + \varepsilon\log12; \\
        m := & \ \frac{1}{[K':\Q]} \sum_{\sigma \in \mathcal{V}} \frac{1}{\delta_\sigma^2}.
    \end{align*}
\end{definition}
	
	By \cite[Part 6]{gaudron-remond},	and in particular \cite[Part 6.8]{gaudron-remond}, we have
 $$\varepsilon\log2 \left( \frac{\varepsilon\theta}{x}m -1
	\right) < \frac{\varepsilon\log2}{[K':\Q]}\sum_{\sigma \in \mathcal{V}} S_\sigma \le \aleph_1 + \frac{\pi x}{2}\left(\frac{3}{2} + \frac{3\theta\varepsilon}{x}\sqrt{m} + \left(\frac{\theta\varepsilon}{x} \right)^2 m\right),$$
	where the inequality on the left is obtained by the definition of $S_\sigma$ together with the inequality $\lfloor x \rfloor > x-1$, while the inequality on the right is that of \cite[Part 6.8, equation (14)]{gaudron-remond} together with the estimate on $\aleph_2$ obtained by the Cauchy-Schwarz inequality on the same page of \cite{gaudron-remond}.
	Solving the inequality in $\sqrt{m}$ we obtain
	\begin{equation}\label{sqrtm-ineq}
		\sqrt{m} < \frac{3\pi x}{2\varepsilon\log2}\left(1 + \sqrt{1+ \frac{8}{9\pi x}\left(\aleph_1 + \frac{3\pi}{4}x + \varepsilon \log2\right)}\right).
	\end{equation}
	
	We now distinguish cases according to whether $p \le \sqrt{2n}$ or $p>\sqrt{2n}$.
	\begin{lemma}
	    If $p \le \sqrt{2n}$, then Theorem \ref{effiso} holds for $E$ and $p$.
	\end{lemma}
	\begin{proof}
		If $p \le \sqrt{2n}$, we can write $$p \le \sqrt{2\lfloor |\tau_{\sigma_0}|^2 \rfloor} \le \sqrt{2|\tau_{\sigma_0}|^2} \le \sqrt{2\left((\Im\{\tau_{\sigma_0}\})^2 + \frac{1}{4}\right)} \le \sqrt{2} \, \Im\{\tau_{\sigma_0}\} + \frac{1}{\sqrt{2}}.$$
	    Remark \ref{rmk: princ pol on ell curves} gives $\Im\{\tau_{\sigma}\} = \rho(E_\sigma, L_\sigma)^{-2}$, so by \cite[Proposition 3.2]{gaudron-remond} we have $$\Im\{\tau_{\sigma_0}\} \le \frac{1}{[K':\Q]}\sum_{\sigma} \Im\{\tau_\sigma\} \le 6.45 \max\left\lbrace\Fheight(E),1\right\rbrace,$$
	    and therefore $$p < 14 \max\left\lbrace\Fheight(E),1\right\rbrace + 15,$$
	    which is largely better than Theorem \ref{effiso} (1) (taking into account Remark \ref{minimalheight}). 
		Concerning part (2), if $K=\Q$, by Theorem \ref{heights-ineq} (3) we have
		\begin{align*}
			|\log|q|| &\le 12\Fheight(E) + 6\log|\log|q|| + 5,
		\end{align*}
		hence
		\begin{align*}
			p &\le \sqrt{2} \, \Im\{\tau\} + \frac{1}{\sqrt{2}} = \frac{1}{\pi\sqrt{2}} |\log|q|| + \frac{1}{\sqrt{2}} \le \frac{6\sqrt{2}}{\pi}\Fheight(E) + 2\log|\log|q|| + 2 \\
			&\le 3\Fheight(E) + \log(\Im\{\tau\}) + 6,
		\end{align*}
		which is again largely better than Theorem \ref{effiso} (2).
	\end{proof}
	\begin{lemma}\label{B[sigma]=0}
		Assume $p> \sqrt{2n}$. Given $A_\omega$ as above, for every $\sigma: K' \hookrightarrow \C$ we have $B[\sigma]=0$, and hence $x=\frac{1}{p\sqrt{2n}}$.
	\end{lemma}
	\begin{proof}
		Since $A_\omega \cong A$, it is sufficient to prove the statement for $A$ and $L=\psi^*L'$. First of all, we notice that $x(0)=\left(\frac{1}{2np^2}\right)^\frac{1}{2}= \frac{1}{p\sqrt{2n}}$. Let us now consider an arbitrary proper abelian subvariety $B$ such that $\dim B >0$. The subgroups of $A$ correspond to those of $E \times E$ that contain $G=\{(x,\gamma \cdot x) \mid x \in E[p]\}$. Since $B$ is a proper subvariety of the abelian surface $A$, we have $\dim B = 1$. Hence, $\varphi^{-1}(B) \subset E \times E$ is an algebraic subgroup of dimension $1$ containing $G$. In particular, there exists an elliptic curve $C \subset E \times E$ such that the algebraic group $\varphi^{-1}(B)$ is $\tilde{C} := \langle C , G \rangle$, and $C$ is the connected component of $\tilde{C}$ that contains $0$. Since $\ker \varphi = G$, we have $\varphi(C) = \varphi(\tilde{C}) = B$, and $C=[p](C)=\psi \circ \varphi (C)= \psi(B)$.
		By assumption, $E$ does not have CM, hence there exist two relatively prime integers $a,b$ such that $$C=\{(P,Q) \in E \times E \mid aP=bQ\}.$$
		Therefore, we have $\deg\varphi|_C = |\ker\varphi|_C| = |C \cap G|$, with $$C \cap G = \left\lbrace  (x,\gamma \cdot x) \mid x \in E[p] \text{ such that } (a-b\gamma)x=0 \right\rbrace.$$
		However, $a-b\gamma$ is always invertible when $a$ and $b$ are coprime (it is a non-zero element in a non-split Cartan subgroup), hence $C \cap G = \{0\}$, and so $\deg\varphi|_C = 1$. We have $$p^2 = \deg[p]|_C = (\deg\psi|_B)(\deg\varphi|_C) = \deg\psi|_B,$$
		and therefore $\deg_LB = \deg_{\psi^*L'}B = (\deg\psi|_B)\deg_{L'}C \ge p^2$.
		We can now estimate
		$$x(B) = \frac{\deg_LB}{\deg_LA} \ge \frac{p^2}{2np^2} = \frac{1}{2n},$$
		and so $x(B) > x(0)$ since $p > \sqrt{2n}$.
	\end{proof}
	\begin{remark}
		As $B[\sigma]=0$ for every $\sigma$, we have $\rho((A_\omega)_\sigma,L_\sigma) = \delta_\sigma$.
	\end{remark}
	
	We now notice that $\rho((A_\omega)_\sigma,L_\sigma) \ge \rho(E_\sigma,(L_E)_\sigma)$, as explained in \cite[Part 7.3, p.~387]{gaudron-remond}, hence $\log \max \left\lbrace 1, \frac{1}{\rho((A_\omega)_\sigma,L_\sigma)} \right\rbrace \le \log \max \left\lbrace 1, \frac{1}{\rho(E_\sigma,(L_E)_\sigma)} \right\rbrace$.
	If we define
	\begin{align*}
		\overline{\aleph}_1 := 2\Fheight(E) + 4\log p + \log \frac{n}{\pi} + 5\log 2 + \frac{2}{[K':\Q]} \sum_{\sigma \in \mathcal{V}} \log \max \left\lbrace 1, \frac{1}{\rho(E_\sigma,(L_E)_\sigma)} \right\rbrace + \varepsilon\log12,
	\end{align*}
	by Lemma \ref{slopes} and Lemma \ref{B[sigma]=0} we have $\aleph_1 \le \overline{\aleph}_1$ (indeed, $\deg_{L_\sigma} 0 = 1$ for every $\sigma$). We can then replace $\aleph_1$ by $\overline{\aleph}_1$ in equation \eqref{sqrtm-ineq}.
	We now treat separately the two cases in which $K$ is $\Q$ or an arbitrary number field.
	
	\subsection{$K=\Q$}
	
	In this case, the field $K'$ of equation \eqref{eq:kappaprimo} is an imaginary quadratic field and the two embeddings of $K'$ are complex conjugate, hence they give the same norm $\|\cdot\|_{L_\sigma} = \|\cdot\|_{L_{\overline{\sigma}}}$. For this reason, we will omit $\sigma$ in the subscripts. Since $E$ is defined over $\Q$, we also have a single $\tau = \tau_\sigma = \tau_{\bar{\sigma}}$. Consider $A_\omega$ as a complex abelian variety via the embedding $\sigma$ and consider $S=S_\sigma$ (see Definition \ref{def: notation for the proof of the effective surjectivity theorem}), which by construction vanishes if and only if
	$$1 > \frac{\theta\varepsilon}{x\delta^2} = \frac{\theta\varepsilon p\sqrt{2n}}{\delta^2} = \frac{\theta\varepsilon p\sqrt{2n}}{\rho(A_\omega,L)^2}.$$
	As shown in \cite[Part 7.3]{gaudron-remond}, $\rho(A_\omega,L)^2 \le \frac{2n}{\sqrt{n-\frac{1}{4}}}$, hence if $S=0$ we have
	$$1 >\frac{\theta\varepsilon p\sqrt{2n}}{\rho(A_\omega,L)^2} \ge \left( \frac{1}{2} - \frac{1}{8n} \right)^\frac{1}{2} \theta\varepsilon p > \sqrt{\frac{3}{8}} \theta\varepsilon p,$$
	that gives $p < \frac{\sqrt{8}}{\theta\varepsilon\sqrt{3}} < 62$, which is better than part 2 of Theorem \ref{effiso} (since $\Fheight(E)+1.94>1$ by Remark \ref{minimalheight}). Thus we can assume $S \ge 1$, and in particular, we can assume that $\mathcal{V}$ contains both the embeddings of $K'$. Then by Remark \ref{rmk: princ pol on ell curves} we have
	\begin{align*}
		\frac{2}{[K':\Q]} \sum_{\sigma \in \mathcal{V}} \log \max \left\lbrace 1, \frac{1}{\rho(E_\sigma,(L_E)_\sigma)} \right\rbrace &= 2\log\max \left\lbrace 1, \frac{1}{\rho(E,L_E)} \right\rbrace \\
		&= \log\max \left\lbrace 1, \frac{1}{\rho(E,L_E)^2} \right\rbrace \\
		&= \max\left\lbrace 0,\log(\Im\{\tau\})\right\rbrace.
	\end{align*}
	Under the assumption $S \ge 1$, we have obtained
	$$\overline{\aleph}_1 = 2\Fheight(E) + 4\log p + \log \frac{n}{\pi} + 5\log2 + \log \max\left\lbrace 1,\Im\{\tau\}\right\rbrace + \varepsilon\log12.$$
	By Remark \ref{minimalheight}, we have $\Fheight(E) \ge -0.75$ for every rational elliptic curve $E$. Since $p- 2 \cdot 2530\log p < 2530(-0.75 + 1.94)$ holds for all primes $p \le 58000$, we can assume $p>58000$, otherwise Theorem \ref{effiso} would trivially hold. Then we have
	$$\overline{\aleph}_1 > -1.5 + 4\log58000 - \log \pi + 5\log2 + \varepsilon\log12 > 44.$$
	Using that $x=\frac{1}{p\sqrt{2n}} \le \frac{1}{p\sqrt{2}}$ (Lemma \ref{B[sigma]=0}), this gives
	$$\frac{8}{9\pi x}\left(\overline{\aleph}_1 + \frac{3\pi}{4}x + \varepsilon \log2\right) > \frac{8 \sqrt{2} \cdot 58000}{9\pi} \cdot 44 > 10^6.$$
	Since the function $\frac{1+\sqrt{z}}{\sqrt{1+z}}$ is decreasing for $z>1$, in equation \eqref{sqrtm-ineq} we obtain
	\begin{align*}
		\frac{1}{\rho(A_\omega,L)} &= \frac{1}{\delta} = \sqrt{m} < \frac{3\pi x}{2\varepsilon\log2}\left(1 + \sqrt{1+ \frac{8}{9\pi x}\left(\overline{\aleph}_1 + \frac{3\pi}{4}x + \varepsilon \log2\right)}\right) \\
		&\le \frac{3\pi x}{2\varepsilon\log2} \cdot \frac{1+1000}{\sqrt{10^6 + 1}} \sqrt{2+ \frac{8}{9\pi x}\left(\overline{\aleph}_1 + \frac{3\pi}{4}x + \varepsilon \log2\right)}.
	\end{align*}
	Squaring both sides and substituting $x=\frac{1}{\sqrt{2n} \, p}$ (Lemma \ref{B[sigma]=0}) we have
	\begin{align*}
		\frac{\sqrt{2n} \, p}{\rho(A_\omega,L)^2} &< \frac{9\pi^2x}{4\varepsilon^2 (\log2)^2} \cdot 1.002 \cdot \left(2+ \frac{8}{9\pi x}\left(\overline{\aleph}_1 + \frac{3\pi}{4}x + \varepsilon \log2\right)\right) \\
		&= \frac{2\pi}{\varepsilon^2(\log2)^2} \cdot 1.002 \left( \overline{\aleph}_1 + 3\pi x + \varepsilon \log2 \right) \\
		&< \frac{2.004\pi}{\varepsilon^2(\log2)^2} \left( \overline{\aleph}_1 + 0.085 \right),
	\end{align*}
	where in the last inequality we bounded $x$ with $\frac{1}{58000\sqrt{2}}$, since we are assuming that $p>58000$. Using again $\rho(A_\omega,L)^2 \le \frac{2n}{\sqrt{n-\frac{1}{4}}}$ and bounding $\overline{\aleph}_1$ we obtain
	\begin{align}\label{daquidueopzioni}
		p &< \left(2- \frac{1}{2n}\right)^{-\frac{1}{2}} \frac{4.008\pi}{\varepsilon^2(\log2)^2} \left( 2\Fheight(E) + 4\log p + \log \frac{n}{\pi} +  \max \{0,\log\Im\{\tau\}\} + 3.86 \right).
	\end{align}
	We now use $\left(2- \frac{1}{2n}\right)^{-\frac{1}{2}} \le \sqrt{\frac{2}{3}}$ and $\log n \le \log|\tau|^2 \le \log\left(\Im\{\tau\}^2 + \frac{1}{4}\right)$. If $|\tau|^2<2$, we have $\log n = 0$, if instead $|\tau|^2 \ge 2$, then $\Im\{\tau\} \ge \frac{\sqrt{7}}{2}$ and so $\log\left(\Im\{\tau\}^2 + \frac{1}{4}\right) \le \log\left(\frac{8}{7}\Im\{\tau\}^2\right) = 2\log(\Im\{\tau\}) + \log(8/7)$. We can combine the two cases by writing $\log \frac{n}{\pi} \le \max\{0,2\log(\Im\{\tau\})\} + \log(8/7) - \log \pi$. Equation \eqref{daquidueopzioni} then yields
	$$p < 2908 \left( \Fheight(E) + 2\log p + \frac{3}{2}\max\{0, \log(\Im\{\tau\})\} + 2 \right),$$
	which is the unconditional inequality of Theorem \ref{effiso}.
	
	If we further assume that $|\log|q||\ge 30$, we can write $n \geq |\tau|^2-1 \ge \Im\{\tau\}^2-1 = \frac{(\log|q|)^2}{4\pi^2}-1 \ge 21.7$. As before, we have the bound
	$$\frac{8}{9\pi x}\left(\overline{\aleph}_1 + \frac{3\pi}{4}x + \varepsilon \log2\right) > \frac{8 \sqrt{44} \cdot 58000}{9\pi} \cdot 44 > 4.7 \cdot 10^6,$$
	and for $z>4.7 \cdot 10^6$ we have $\left(\frac{1+\sqrt{z}}{\sqrt{z+1}}\right)^2 < 1.00093.$
	We can then estimate $\left(2- \frac{1}{2n}\right)^{-\frac{1}{2}} \le \sqrt{\frac{44}{87}}$ and $\log \frac{n}{\pi} \le 2\log(\Im\{\tau\}) + \log\left(1+ \frac{1}{4\Im\{\tau\}^2}\right) - \log \pi \le 2\log(\Im\{\tau\}) + 0.011 - \log \pi$. Hence, since $\Im\{\tau\} > 1$, equation \eqref{daquidueopzioni} (with $4.00372$ instead of $4.008$) yields
	$$p < 2530 \left( \Fheight(E) + 2\log p + \frac{3}{2}\log(\Im\{\tau\}) + 1.94 \right),$$
	which concludes the proof of Theorem \ref{effiso} (2).
	
	\subsection{$K$ is a generic number field}
	
	By assumption, we have $\Fheight(E) > 1$. Moreover, we can also assume that $[K:\Q]>1$, since if $[K:\Q]=1$ the second part of Theorem \ref{effiso} implies the first. Similarly to the case $K=\Q$, we can assume $p> 170000$, since otherwise the claim is trivial. 
Following \cite[Part 7.3]{gaudron-remond} and using \cite[Proposition 3.2]{gaudron-remond} we have
	$$\frac{2}{[K':\Q]} \sum_{\sigma \in \mathcal{V}} \log \max \left\lbrace 1, \frac{1}{\rho(E_\sigma,(L_E)_\sigma)} \right\rbrace \le \log\left(\Fheight(E) \right) + \log6.45.$$
	We define
	$$\widetilde{\aleph}_1 := 2\Fheight(E) + 4\log p + \log \frac{n}{\pi} + \log\left(\Fheight(E) \right) + 5.6313,$$
	which satisfies $\widetilde{\aleph}_1 > \overline{\aleph}_1 \ge \aleph_1$, so we can again write $\widetilde{\aleph}_1$ instead of $\aleph_1$ in equation \eqref{sqrtm-ineq}. As before, we can bound $\widetilde{\aleph}_1$ from below writing
	$$\widetilde{\aleph}_1 > 2 + 4\log170000 -\log\pi  + 5.6313 > 54.6,$$
	and so 
	$$\frac{8}{9\pi x}\left(\widetilde{\aleph}_1 + \frac{3\pi}{4}x + \varepsilon \log2\right) > \frac{8 \sqrt{2} \cdot 170000}{9\pi} \cdot 54.6 > 3.7 \cdot 10^6.$$
	This gives
	$$\sqrt{m} < \frac{3\pi x}{2\varepsilon\log2} \cdot \frac{1+1000\sqrt{3.7}}{\sqrt{3.7 \cdot 10^6 + 1}} \sqrt{2+ \frac{8}{9\pi x}\left(\widetilde{\aleph}_1 + \frac{3\pi}{4}x + \varepsilon \log2\right)},$$
	and by the same arguments as in the case $K=\Q$ we obtain
	\begin{align*}
		p &< \left(2- \frac{1}{2n}\right)^{-\frac{1}{2}} \frac{[K':\Q]}{2}\frac{4.0042\pi}{\varepsilon^2(\log2)^2} \left( \widetilde{\aleph}_1 + 3\pi x + \varepsilon\log2 \right).
	\end{align*}
	By \cite[Proposition 3.2]{gaudron-remond} we have $$n \le |\tau_{\sigma_0}|^2 \le \Im\{\tau_{\sigma_0}\}^2 + \frac{1}{4} \le 6.45^2 \left(\Fheight(E) + \frac{1}{2} \log\pi \right)^2 + \frac{1}{4},$$
	and so
	$$\log n \le 2\log\left(\Fheight(E) + \frac{1}{2}\log\pi \right) + 3.7342.$$
	Putting everything together we obtain
	$$p < 2905 [K:\Q] \left( \Fheight(E) + 2\log p + \frac{3}{2}\log\left(\Fheight(E) + \frac{1}{2}\log\pi\right) + 4.725 \right),$$
	concluding the proof.

\section{Modular units and an upper bound on $\log|q|$}\label{sec:modular}

The aim of this section is to prove Proposition \ref{logq<30}, which gives the absolute upper bound $\log |j(E)| \leq 39 + \log 2$ for all elliptic curves $\faktor{E}{\Q}$ which satisfy $\operatorname{Im} \rho_{E, p} \cong G(p)$ for some prime $p>5$. This should be contrasted with the estimate $\log |j(E)| \leq 27000$ given in \cite{lefournlemos21}. 

For technical reasons, in the whole section we work with the quantity $\log |q|$ instead of $\log |j(E)|$, where $q=e^{2\pi i \tau}$ and $\tau$ is a point in the upper half plane $\uhp$ corresponding to $E(\C)$. By Theorem \ref{estimate-qj}, whenever $\tau$ is in the standard fundamental domain $\mathcal{F}$, estimates on $\log |j(E)|$ translate into estimates on $\log |q|$ and vice versa.

The improved bound is obtained in two steps.
In Proposition \ref{logq<p}, we obtain a preliminary bound on $|\log|q||$ which is already sharper than \cite[Proposition 6.1]{lefournlemos21} ($O(\sqrt[4]{p})$ instead of $O(\sqrt{p})$, with the key improvement given by Lemma \ref{kloosterman}). 
This allows us to prove that $p<\boundonprimes$: we then use this to re-estimate $\log|q|$ and obtain the final bound $|\log|q||<39$.

\subsection{Modular Units}

As in \cite{lefournlemos21}, we recall certain modular units studied in \cite{modunits}. In particular, we give a summary of \cite[§6]{lefournlemos21}: we define certain modular units for the inverse image of $G(p)$ in $\operatorname{SL}_2(\Z)$ and describe some of their properties.
\begin{definition}
	Let $\tau \in \uhp$. For all $(a_1,a_2) \in \frac{1}{p}\Z^2 \cap [0,1)^2$, with $a_1,a_2$ not both $0$, we define
	$$g_{a_1,a_2}=q^{\frac{B_2(a_1)}{2}}e(a_2(a_1-1)/2) \prod_{n=0}^{\infty}(1-q^{n+a_1}e(a_2))(1-q^{n+1-a_1}e(-a_2)),$$
	where $q^k=e^{2\pi i k\tau}$ and $e(k)=e^{2\pi i k}$ for any $k \in \Q$, and $B_2(x)=x^2-x+\frac{1}{6}$.
\end{definition}

\begin{remark}\label{rmk: wrong root of unity}
    See \cite[Chapter 2, Section 1, Equation K 4 on page 29]{modunits} for this definition of the Siegel function $g_{a_1, a_2}$. The definition of $g_{a_1,a_2}$ in \cite[equation (6.1)]{lefournlemos21} is slightly imprecise, as the term $e(a_2(a_1-1)/2)$ is written as $e(a_2(a_1-1))$. However, this does not affect any of the arguments in \cite{lefournlemos21}.
\end{remark}

\begin{definition}
	We set $\mathcal{O}_{\cubes} := \left\lbrace \begin{pmatrix} a \\ b \end{pmatrix} \in \F_p^2 \setminus \{0\} \ \middle| \ a+b\sqrt{\varepsilon} \in \F_{p^2}^{\times 3}  \right\rbrace$, where $\varepsilon$ is the same as in equation \eqref{eq:cartan} and $\sqrt{\varepsilon}$ is a fixed element in $\mathbb{F}_{p^2}$ with square $\varepsilon$. We define
	$$U(\tau)= \zeta \prod_{(a,b) \in \mathcal{O}_{\cubes}} g_{\frac{a}{p}, \frac{b}{p}}^3,$$
	where $\zeta$ is a root of unity such that the coefficient of the lowest power of $q$ in $U$ is 1.
\end{definition}

\begin{remark}
	As explained in \cite[Lemma 6.3]{lefournlemos21}, the quotient $\faktor{\mathcal{O}_{\cubes}}{\pm 1}$ parametrises the Galois orbit of the cusps $\infty$ of the modular curve $X_{G(p)}$.
\end{remark}

The next two results are both contained in \cite[Proposition 6.4]{lefournlemos21}.
\begin{theorem}[Le Fourn, Lemos]
	We have the following:
	\begin{itemize}
		\item $U^2 \in \Q(X_{G(p)})$.
		\item The zeroes of $U^2$ are the cusps at infinity (that is, the Galois orbit of the cusp $\infty$), while its poles are the other cusps.
		\item Both $U^2$ and $\frac{p^6}{U^2}$ are integral over $\Z[j]$.
	\end{itemize}
\end{theorem}

\begin{corollary}\label{Uint}
	For every $P \in X_{G(p)}(\Q)$ such that $j(P)$ is an integer, $U^2(P)$ is an integer dividing $p^6$. In particular, $0 \le \log|U^2(P)| \le 6\log p$, or equivalently, $0 \le \log|U(P)| \le 3\log p$.
\end{corollary}

\begin{remark}
    In \cite{lefournlemos21}, both these results are stated for $U$ rather than $U^2$. However, the root of unity in front of our $U(\tau)$ is not the same as in \cite{lefournlemos21} (see Remark \ref{rmk: wrong root of unity}), and we were unable to confirm that $U$ is defined over $\Q$. However, an immediate application of \cite[Theorem 5.5]{bilu21} (taking $m=6$) shows that the conclusion holds for $U^2$, and the proof of Corollary \ref{Uint} goes through.
\end{remark}

We now introduce the auxiliary quantities that we will have to bound in our proof.

\begin{definition}\label{def: Ra1a2 and R}
We define functions $R_{a_1, a_2}=R_{a_1,a_2}(q)$ as follows. For all $(a_1,a_2) \in \frac{1}{p}\Z^2 \cap [0,1)^2$, with $a_1,a_2$ not both $0$, we define
$$R_{a_1,a_2}=\prod_{n=0}^{\infty}(1-q^{n+a_1}e(a_2))(1-q^{n+1-a_1}e(-a_2))$$ for $a_1 \ne 0$, and $$R_{0,a_2}=\prod_{n=1}^{\infty}(1-q^{n}e(a_2))(1-q^{n}e(-a_2)).$$
We further set
\[
R= \prod_{(pa_1,pa_2) \in \mathcal{O}_{\cubes}} R_{a_1,a_2}^3,
\]
where we identify $pa_1, pa_2 \in [0, p-1] \cap \mathbb{Z}$ with their residue classes modulo $p$.
\end{definition}

\begin{remark}\label{rmk: norm of constant term}
	We have $\prod_{a_2=1}^{p-1} (1-e(a_2))=p$, because $\prod_{a_2=1}^{p-1}(x-e(a_2))=1+x+x^2+ \ldots +x^{p-1}$.
\end{remark}

\begin{remark}\label{Fpcubes}
	Whenever $p \equiv 2 \pmod3$, we have $\F_p^\times = \F_p^{\times 3}$, and so $\F_p^\times \subseteq \F_{p^2}^{\times 3}$.
	This implies that $(0,a_2) \in \mathcal{O}_{\cubes}$ for every $a_2 \in \F_p^\times$, because $a_2\sqrt{\varepsilon}=\frac{a_2}{\varepsilon}\cdot  \sqrt{\varepsilon}^3$ is a cube in $\F_{p^2}$, since it is the product of two cubes.
\end{remark}

The last two remarks imply that when $p \equiv 2 \pmod3$ we can write $U=\zeta \cdot q^{\operatorname{Ord}_q(U)} \cdot p^3 \cdot R$, hence 
\begin{equation}\label{eq: use U to estimate log |q|}
\log|U|= \operatorname{Ord}_q(U) \log |q|+ 3 \log p + \log |R|.
\end{equation}

\subsection{Comparing $\log|q|$ with $p$}

Our next goal is to establish the following bound on $\log|q|$ in terms of $p$.

\begin{proposition}\label{logq<p}Let $\faktor{E}{\Q}$ be an elliptic curve and set $q=e^{2\pi i \tau}$, where $\tau \in \uhp$ 
corresponds to the complex elliptic curve $E(\mathbb{C})$. Let $p>5$ be a prime number such that $p \equiv 2 \pmod 3$ and $\operatorname{Im} \rho_{E,p}$ is conjugate to $G(p)$. Assume furthermore $|\log|q|| \geq 30$. We have $$|\log|q|| \le \frac{2\sqrt{2} \, \pi \cdot 101}{10 \sqrt{102}} \cdot \sqrt[4]{p} + 1.65.$$
\end{proposition}

The proof of this result will occupy all of this section. The argument relies on estimating the various terms in equation \eqref{eq: use U to estimate log |q|}. In particular, we need to compute the order at infinity of $U$ and bound the contribution of $\log |R|$. The latter is the hard step; we take care of the former in the next lemma.

\begin{lemma}\label{order}
	We have $$\operatorname{Ord}_q(U)= \frac{p^2-1}{6p}.$$
\end{lemma}

The calculation of $\operatorname{Ord}_q(U)$ already appears in \cite[Proposition 6.5]{lefournlemos21}, but unfortunately, due to an arithmetic error, the result is incorrect. For the sake of completeness, we repeat the calculation below. Note that the second half of \cite[Proposition 6.5]{lefournlemos21}, namely the statement that $|\rho_U|=(p-1)^3$, is also incorrect: by Remark \ref{rmk: norm of constant term} we actually have $|\rho_U|=p^3$.

\begin{proof}
	By Remark \ref{Fpcubes}, we know that $(0,a_2) \in \mathcal{O}_{\cubes}$ for every $a_2 \in \F_p^\times$.  If instead $a_1 \in \F_p^\times$, the function $(a_1,a_2) \mapsto a_1$ has fibers with constant cardinality, because $\mathcal{O}_{\cubes}$ is stable under multiplication by $\F_p^\times$. In particular, the cardinality of each fiber is $\frac{(p^2-1)/3-(p-1)}{p-1}=\frac{p-2}{3}$. Hence
	\begin{align*}
		\operatorname{Ord}_q(U)= & 3\left((p-1) \cdot \frac{1}{2} B_2(0) + \frac{p-2}{3} \sum_{a_1=1}^{p-1} \frac{1}{2} B_2\left(\frac{a_1}{p}\right) \right) = \frac{p^2-1}{6p}. \qedhere
	\end{align*}
\end{proof}

Our next objective is to estimate $\log |R|$. 

\begin{proposition}\label{logR}
	We have
	$$|\log|R|| \le  - \frac{4\pi^2 p \sqrt{p}}{3\log|q|}. $$
\end{proposition}

\begin{proof}

The inequality $|\log |z|| \le |\log z|$ holds for every $z \in \C^\times$ and every choice of a branch of the logarithm. Indeed, if $z=r \cdot e^{i\theta}$, we have $|\log |z|| = |\log r| \le |\log r + i\theta +2k\pi i| = |\log z|$.
	Thus, it suffices to bound $|\log R|$.
	As $(a,b)$ is in $\mathcal{O}_{\cubes}$ if and only if $(-a,-b)$ is, we have
	\begin{align*}
		R = \prod_{(pa_1,pa_2) \in \mathcal{O}_{\cubes}} R_{a_1,a_2}^3 = \left( \prod_{b=1}^{p-1} \prod_{n=1}^\infty (1-q^ne(b/p))^6 \right) \cdot \left( \prod_{\substack{(pa_1,pa_2) \in \mathcal{O}_{\cubes} \\ a_1 \ne 0}} \prod_{n=0}^\infty (1-q^{n+a_1}e(a_2))^6 \right).
	\end{align*}
We can further write
	\begin{align*}
		\log R &= 3\sum_{(pa_1,pa_2) \in \mathcal{O}_{\cubes}} \log R_{a_1,a_2} \\
		&= 3\sum_{b=1}^{p-1} \log R_{0,\frac{b}{p}} + 3\sum_{\substack{(pa_1,pa_2) \in \mathcal{O}_{\cubes} \\ a_1 \ne 0}} \log R_{a_1,a_2} \\
		&= 6\sum_{b=1}^{p-1} \sum_{n=1}^\infty \log(1-q^n e(b/p)) + 6\sum_{\substack{(pa_1,pa_2) \in \mathcal{O}_{\cubes} \\ a_1 \ne 0}} \sum_{n=0}^\infty \log (1-q^{n+a_1}e(a_2)) \\
		&= -6\sum_{b=1}^{p-1} \sum_{n=1}^\infty \sum_{k=1}^\infty \frac{q^{kn}}{k} \cdot e(kb/p) - 6\sum_{\substack{(pa_1,pa_2) \in \mathcal{O}_{\cubes} \\ a_1 \ne 0}} \sum_{n=0}^\infty \sum_{k=1}^\infty \frac{q^{k(n+a_1)}}{k} \cdot e(ka_2).
	\end{align*}
 	Define now $c(a):=\sum_{b \in F(a)} e(b/p)$ with $F(a):=\{b \in \F_p \mid (a,b) \in \mathcal{O}_{\cubes}\}$ for $a \not \equiv 0 \pmod{p}$. We extend the definition to $a \equiv 0 \pmod{p}$ by setting $c(a)=c(0):=\frac{p-2}{3}$. 
	
 The sum $\sum_{b=1}^{p-1} e(kb/p)$ equals either $-1$ or $p-1$ if respectively $k \not\equiv 0 \pmod{p}$ or $k \equiv 0 \pmod{p}$. Moreover, we also have $\sum_{b \in F(a)} e(kb/p)=c(ka)$: indeed, $b$ is in $F(a)$ if and only if $kb$ is in $F(ka)$, because $k$ is an element of $\F_p^\times$ and therefore a cube in $\F_{p^2}^\times$. Hence we obtain
	\begin{align*}
		\log R &= 6\sum_{n=1}^\infty \sum_{k \not\equiv 0(p)} \frac{q^{kn}}{k} - 6(p-1)\sum_{n=1}^\infty \sum_{\substack{k \equiv 0(p) \\ k>0}} \frac{q^{kn}}{k} - 6\sum_{a=1}^{p-1} \sum_{n=0}^\infty \sum_{k=1}^\infty \frac{q^{k(n+\frac{a}{p})}}{k} \cdot c(ka) \\
		&= 6\sum_{n=1}^\infty \sum_{k=1}^\infty \frac{q^{kn}}{k} - 6p\sum_{n=1}^\infty \sum_{\substack{k \equiv 0(p) \\ k>0}} \frac{q^{kn}}{k} - 6\sum_{a=1}^{p-1} \sum_{n=0}^\infty \sum_{k=1}^\infty \frac{q^{k(n+\frac{a}{p})}}{k} \cdot c(ka).
	\end{align*}
	We notice that the definition we have given for $c(0)$ is compatible with this chain of equalities. Indeed, whenever $k \equiv 0 \pmod{p}$, for a fixed $a_1 \ne 0$ we have $\sum_{a_2 \in F(a_1)} e(0 \cdot a_2) = |F(a_1)| = \frac{p-2}{3} = c(0)$, as we noticed at the beginning of the proof of Lemma \ref{order}.
	
	\begin{lemma}\label{kloosterman}
		For every $s \in \F_p^\times$ we have $|c(s)| \le \frac{4}{3}\sqrt{p}$.
	\end{lemma}
	\begin{proof}
		For $a+b\sqrt{\varepsilon} \in \F_{p^2}$ we have
		$$(a+b\sqrt{\varepsilon})^3 = a^3+3\varepsilon ab^2 + (3a^2b+\varepsilon b^3)\sqrt{\varepsilon}.$$
		To characterise the set $F(s)$ we write $a^3+3\varepsilon ab^2 = s$. Note that, since $s \ne 0$, we always have $a \ne 0$. Writing $e_p(x):=e(x/p)$ and $t=\frac{b}{a}$, this gives
		\begin{align*}
			c(s)&= \sum_{x \in F(s)} e_p(x) = \frac{1}{3} \sum_{\substack{a,b \in \F_p \\ a^3+3\varepsilon ab^2 = s}} e_p(3a^2b+\varepsilon b^3) \\
			&= \frac{1}{3} \sum_{\substack{a,t \in \F_p \\ a^3(1+3\varepsilon t^2) = s}} e_p(a^3(3t+\varepsilon t^3)) = \frac{1}{3} \sum_{\substack{t \in \F_p \\ 1+3\varepsilon t^2 \ne 0}} e_p\left(\frac{s(3t+\varepsilon t^3)}{1+3\varepsilon t^2}\right),
		\end{align*}
		where the second and last equalities are due to the fact that, for $c \in \F_p^\times$, the equation $z^3=c$ has 3 solutions in $\F_{p^2}$ and 1 solution in $\F_p$ (since $p \equiv 2 \pmod{3})$. We now use the following result by Perel'muter \cite{perel}, obtained via a generalisation of Weil's strategy \cite{weil48} to bound Kloosterman sums.
  
		Let $\varphi \in \F_p(t)$ be a rational function with poles $S=\{t_1 , \dots , t_\ell\} \subseteq \overline{\F}_p \cup \{\infty\}$. We have
		$$\left| \sum_{t \in \F_p \setminus S} e_p(\varphi(t)) \right| \le (\ell + \deg(\varphi) -2) \sqrt{p}.$$
		In our case we have that $\varphi(t)=\frac{s(3t+\varepsilon t^3)}{1+3\varepsilon t^2}$ has 2 poles other than $\infty$, hence we get
		$$|c(s)| \le \frac{1}{3} (3 + 3 - 2) \sqrt{p} = \frac{4}{3} \sqrt{p},$$
		as desired.
	\end{proof}
	Thanks to this lemma, we now have all the tools needed to complete the proof of Proposition \ref{logR}. We notice that
	$$\sum_{a=1}^{p-1} \sum_{n=0}^\infty \sum_{k=1}^\infty \frac{q^{k(n+\frac{a}{p})}}{k} \cdot c(ka) = \sum_{n \not\equiv 0 (p)} \sum_{k=1}^\infty \frac{q^{\frac{nk}{p}}}{k} \cdot c(kn).$$
	Isolating the terms involving $c(0)$ and using $c(0)=\frac{p-2}{3}$, we can rearrange the sums as follows:
	\begin{align*}
		\log R&= 6\sum_{n=1}^\infty \sum_{k=1}^\infty \frac{q^{kn}}{k} - 6\sum_{n=1}^\infty \sum_{\substack{k \equiv 0(p) \\ k>0}} \frac{q^{\frac{k}{p} \cdot pn}}{k/p} - 6\sum_{n \not\equiv 0 (p)} \sum_{k=1}^\infty \frac{q^{\frac{nk}{p}}}{k} \cdot c(kn) \\
		&= 6\left( \sum_{n=1}^\infty \sum_{k=1}^\infty \frac{q^{kn}}{k} - \sum_{\substack{n \equiv 0(p) \\ n>0}} \sum_{k=1}^\infty \frac{q^{kn}}{k} \right) - 6\left( \sum_{n \not\equiv 0 (p)} \sum_{k=1}^\infty \frac{q^{\frac{nk}{p}}}{k} \cdot c(kn) \right) \\
		&= 6 \sum_{n \not\equiv 0(p)} \sum_{k=1}^\infty \frac{q^{kn}}{k} - 6\left( \sum_{n \not\equiv 0 (p)} \sum_{k \not\equiv 0 (p)} \frac{q^{\frac{nk}{p}}}{k} \cdot c(kn) + \frac{p-2}{3}\sum_{n \not\equiv 0 (p)} \sum_{k=1}^\infty \frac{q^{nk}}{pk} \right)  \\
		&= \alpha \sum_{n \not\equiv 0(p)} \sum_{k=1}^\infty \frac{q^{kn}}{k} - 6\sum_{n \not\equiv 0 (p)} \sum_{k \not\equiv 0 (p)} \frac{q^{\frac{nk}{p}}}{k} \cdot c(kn) + 8\sqrt{p}\sum_{n \not\equiv 0 (p)} \sum_{k=1}^\infty \frac{q^{nk}}{pk},
	\end{align*}
	where $\alpha = 6 -\frac{2(p-2)}{p}-\frac{8\sqrt{p}}{p}$. We now notice that $\alpha \le 4$ for all $p$ and apply Lemma \ref{kloosterman} to estimate $\log R$:
	\begin{align*}
		|\log R| &\le 4\sum_{n \not\equiv 0(p)} \sum_{k=1}^\infty \frac{|q|^{kn}}{k} + 6\sum_{n \not\equiv 0 (p)} \sum_{k \not\equiv 0(p)} \frac{|q|^{\frac{nk}{p}}}{k} \cdot |c(kn)| + 8\sqrt{p}\sum_{n \not\equiv 0 (p)} \sum_{\substack{k \equiv 0(p) \\ k>0}} \frac{|q|^{\frac{nk}{p}}}{k} \\
		&\le 4\sum_{n=1}^\infty \sum_{k=1}^\infty \frac{|q|^{kn}}{k} + 8 \sqrt{p} \sum_{n \not\equiv 0 (p)} \sum_{k \not\equiv 0 (p)} \frac{|q|^{\frac{nk}{p}}}{k} + 8\sqrt{p}\sum_{n \not\equiv 0 (p)} \sum_{\substack{k \equiv 0(p) \\ k>0}} \frac{|q|^{\frac{nk}{p}}}{k} \\
		&= 4\sum_{n=1}^\infty \sum_{k=1}^\infty \frac{|q|^{kn}}{k} + 8 \sqrt{p} \sum_{n \not\equiv 0 (p)} \sum_{k=1}^\infty \frac{|q|^{\frac{nk}{p}}}{k} \\
		&= 4\sum_{n=1}^\infty \sum_{k=1}^\infty \frac{|q|^{kn}}{k} - 8\sqrt{p}\sum_{n=1}^\infty \sum_{k=1}^\infty \frac{|q|^{kn}}{k} + 8 \sqrt{p} \sum_{n=1}^\infty \sum_{k=1}^\infty \frac{|q|^{\frac{nk}{p}}}{k} \\
		&= (8\sqrt{p}-4) \sum_{n=1}^\infty \log(1-|q|^n) - 8\sqrt{p} \sum_{n=1}^\infty \log(1-|q|^\frac{n}{p}).
	\end{align*}
	To complete the proof, it suffices to notice that $$\sum_{n=1}^\infty \log(1-|q|^n) < 0$$ and that $$-8\sqrt{p} \sum_{n=1}^\infty \log(1-|q|^\frac{n}{p}) \le - \frac{8\pi^2 \sqrt{p}}{6 \log(|q|^\frac{1}{p})} \le - \frac{4\pi^2 p\sqrt{p}}{3 \log|q|}$$ by Lemma \ref{logsum}.
\end{proof}

Let now $\tau_P$ be a point in the fundamental domain of $X_{G(p)}$ that corresponds to a point $P \in X_{G(p)}(\Q)$ with $j(P) \in \Z$. There exists $\gamma \in \operatorname{SL}_2(\Z)$ such that $\tau = \gamma^{-1}(\tau_P)$ is in the standard fundamental domain $\mathcal{F}$ of $X(1)$; in particular, it is in the domain corresponding to the cusp $\infty$ (i.e., $\infty$ is the cusp closest to $\tau$). We remark that $(U \circ \gamma)(\tau) \in \Z$, since by Corollary \ref{Uint} we have $U(P) \in \Z$.

Up to Galois conjugation (which fixes $P$ but changes the cusps), we can choose an embedding $X_{G(p)}(\overline{\Q}) \hookrightarrow X_{G(p)}(\C)$ such that either $\gamma$ is the identity or $\gamma \pmod p$ is an element in $C_{ns}(p) \cap \operatorname{SL}_2(\F_p)$ that does not lie in $G(p)$. Indeed, this can be seen from the parametrisation of the cusps given in \cite[Section 2]{lefournlemos21} and the fact that the cusps of $X_{G(p)}$ split into two Galois orbits, see \cite[Lemma 6.3]{lefournlemos21}. From now on, whenever we write $\gamma$ we will refer to the second case, in which $\gamma \pmod p$ is an element of $C_{ns}(p) \cap \operatorname{SL}_2(\F_p)$ not in $G(p)$, unless otherwise specified.

\begin{remark}
	By Remark \ref{Fpcubes}, we have that $(0,b) \in \mathcal{O}_{\cubes}$ for every $b \in \F_p^\times$, hence every cusp in $\gamma^{-1} \mathcal{O}_{\cubes}$ is parametrised by a pair $(a,b)$ such that $a \ne 0$. The function $U \circ \gamma$ is a modular unit on $X_{G(p)}$ (though not necessarily defined over $\Q$), and the element $\gamma$ acts by permutation on the set $\mathbb{F}_{p}^2 \setminus \{0\}$. From this, it is easy to see that we have
	$$\log|U \circ \gamma| = \operatorname{Ord}_q(U \circ \gamma) \log|q| + \log |R_{\gamma}|,$$
	where 
 \begin{equation}\label{eq: definition of R_gamma}
 R_{\gamma}=\prod_{(a,b) \in \gamma^{-1}\mathcal{O}_{\cubes}} R_{\frac{a}{p},\frac{b}{p}}^3.
 \end{equation}
\end{remark}

\begin{lemma}\label{ordergamma}
	We have $$\operatorname{Ord}_q(U \circ \gamma) = -\frac{p^2-1}{12p}.$$
\end{lemma}

This is proven by a calculation analogous to that of Lemma \ref{order}. The result also appears in \cite{lefournlemos21}, where however it is affected by the same arithmetic error as \cite[Proposition 6.5]{lefournlemos21}. The next proposition bounds $\log R_\gamma$ similarly to Proposition \ref{logR}; we will not directly make use of this result, but some of the arguments in its proof will be useful later.

\begin{proposition}\label{prop:lavera_logRgamma}
    We have $$|\log|R_\gamma|| < - \frac{4\pi^2 p \sqrt{p}}{3\log|q|}.$$
\end{proposition}

\begin{proof}

	The proof is analogous to that of Proposition \ref{logR}. We notice that for every $(a_1,a_2) \in \gamma^{-1} \mathcal{O}_{\cubes}$ we have $a_1 \ne 0$, hence
	$$\log R_\gamma = -6\sum_{n \not\equiv 0 (p)} \sum_{k=1}^\infty \frac{q^{\frac{nk}{p}}}{k} \cdot c_{\gamma}(kn),$$ where $c_\gamma(a):=\sum_{b \in F_\gamma(a)} e(b/p)$ and $F_\gamma(a):=\{b \in \F_p \bigm\vert (a,b) \in \gamma^{-1}\mathcal{O}_{\cubes}\}$ for $a \ne 0$ and $c(0)=\frac{p+1}{3}$.
	To adapt Lemma \ref{kloosterman} to the case of $c_\gamma$, we notice that if $\gamma^{-1}$ acts as multiplication by $x+y\sqrt{\varepsilon}$ on $\F_{p^2}^\times$ for some $x,y \in \F_p$ (as explained in \cite[Lemma 6.3]{lefournlemos21}), then the function $\varphi(t)$ of Lemma \ref{kloosterman} becomes
	$$\varphi(t)= s \ \frac{(3t+\varepsilon t^3)x + (1+3\varepsilon t^2)y}{(1+3\varepsilon t^2)x + (3t+\varepsilon t^3)\varepsilon y},$$
	giving again $|c_\gamma(s)| \le \frac{4}{3} \sqrt{p}$.
	The rest of the proof of Proposition \ref{logR} carries through.
\end{proof}

We remark that, even though the bound on $|\log|R_\gamma||$ is the same as that on $|\log|R||$, the order at infinity of the function $U$ is halved in this case, that is, $|\operatorname{Ord}_q(U \circ \gamma)| = \frac{1}{2}|\operatorname{Ord}_qU|$. This leads to a weaker bound on $|\log |q||$ in terms of $p$, which is what we are really interested in for the proof of Proposition \ref{logq<p}. To obtain a sharper bound on $|\log|R_\gamma||$, we consider a different $p$-th root of $q$. 
\begin{remark}\label{rmk: qinfunddom}
    We note that in order to prove Proposition \ref{logq<p}, it suffices to consider $\tau \in \uhp$ lying in the standard fundamental domain $\mathcal{F}$ and such that $|\tau|>1$. Indeed, $q(\tau)=q(\tau+n)$ for every $n \in \Z$, hence without loss of generality we can consider $\Re\tau \in \left(-\frac{1}{2}, \frac{1}{2}\right]$, and if $|\tau| \le 1$ then $|\log|q|| \le 2\pi < 30$, contradicting the hypothesis.
\end{remark}
From Theorem \ref{qisreal} we know that if $E$ corresponds to $\tau \in \uhp$ in the standard fundamental domain $\mathcal{F}$ not lying on the lower boundary $\left\lbrace e^{i\theta} \bigm\vert \frac{\pi}{3} \le \theta \le \frac{\pi}{2} \right\rbrace$, then $q \in \mathbb{R}$. This is true for all the fundamental domains of the form $\mathcal{F}+n$ for $n \in \Z$. We notice that, for $\tau \in \mathcal{F}$ and $q>0$, we have $\Re\tau=0$ and therefore $q^\frac{1}{p} \in \mathbb{R}$. However, if $q<0$, we have $\Re\tau=\frac{1}{2}$ and $q^{\frac{1}{p}}=e^{\frac{2\pi i \tau}{p}}$ is not real. However, we can consider $\tau' := \tau + \frac{p-1}{2} \in \mathcal{F} + \frac{p-1}{2}$, which gives the same value of $q$ and is such that $e^{\frac{2 \pi i \tau'}{p}} \in \mathbb{R}$ is the $p$-th real root of $q$.

We then repeat the previous construction changing the choice of $\tau$. Let $\tau_P$ be a point in the fundamental domain of $X_{G(p)}$ that corresponds to a point $P \in X_{G(p)}(\Q)$ with $j(P) \in \Z$. There exists $\gamma \in \operatorname{SL}_2(\Z)$ such that $\tau = \gamma^{-1}(\tau_P)$ is in the standard fundamental domain $\mathcal{F}$ if $q>0$, and in the fundamental domain $\mathcal{F}+\frac{p-1}{2}$ if $q<0$.
As before, up to Galois conjugation, we can take $\gamma$ to be either the identity or an element whose reduction modulo $p$ lies in $C_{ns}(p) \cap \operatorname{SL}_2(\F_p)$ but not in $G(p)$ -- this is again because the point $P$ is defined over $\Q$ and therefore fixed by the Galois action, while there are two orbits of cusps.

All the previous estimates still hold for this new choice of $\tau$ and we can take advantage of the new choice of the $p$-th root of $q$ to improve the bound on $\log|R_\gamma|$. To distinguish the two different $p$-th roots, we will write $q^{\left(\frac{1}{p}\right)}$ to denote the root that maps the real numbers to themselves.

\begin{proposition}\label{logRgamma}
	Let $\tau = \gamma^{-1}\tau_P \in \uhp$ be as above and such that $j(\tau) \not\in (0,1728)$. We have $$\log|R_\gamma(\tau)| = -\frac{1}{2} \log R(\tau).$$
\end{proposition}

\begin{proof}
	As in Proposition \ref{prop:lavera_logRgamma} we have
	$$\log R_\gamma(\tau) = -6\sum_{n \not\equiv 0 (p)} \sum_{k=1}^\infty \frac{q^{\left(\frac{nk}{p}\right)}}{k} \cdot c_{\gamma}(kn),$$ where $c_\gamma(a):=\sum_{b \in F_\gamma(a)} e(b/p)$ and $F_\gamma(a):=\{b \in \F_p \mid (a,b) \in \gamma^{-1}\mathcal{O}_{\cubes}\}$ for $a \ne 0$ and $c(0)=\frac{p+1}{3}$. \\
	As explained in \cite[Lemma 6.3]{lefournlemos21}, the action of $\gamma^{-1}$ on the cusps of $X_{G(p)}$ corresponds to the multiplication by $(x+y\sqrt{\varepsilon})$ on $\faktor{\F_{p^2}^\times}{\pm 1}$ for some $x,y \in \F_p$. It is then easy to see that $\mathcal{O}_{\cubes} \sqcup \gamma^{-1}\mathcal{O}_{\cubes} \sqcup \gamma^{-2}\mathcal{O}_{\cubes} = \F_p^2 \setminus \{(0,0)\}$, and that $(a,b) \in \gamma^{-1}\mathcal{O}_{\cubes}$ if and only if $(a,-b) \in \gamma^{-2}\mathcal{O}_{\cubes}$. Therefore, we obtain that $c_\gamma(k)=\overline{c_{\gamma^2}(k)}$ and $c(k) + c_\gamma(k) + c_{\gamma^2}(k) = 0$ for every $k \in \F_p^\times$. This implies that $c(k)$ is real (this can also be seen directly from the definition of $\mathcal{O}_{\cubes}$) and that $\Re\{c_\gamma(k)\} = -\frac{1}{2}c(k)$ for every $k \not\equiv 0 \pmod p$. \\
	Using that $q \in \mathbb{R}$ by Theorem \ref{qisreal} and that $\log|z|=\Re\{\log z\}$ for every $z \in \C^\times$, we have
	\begin{align*}
		\log |R_\gamma(\tau)| &= \Re\{\log R_\gamma(\tau)\} = -6\sum_{n \not\equiv 0 (p)} \sum_{k=1}^\infty \frac{q^{\left(\frac{nk}{p}\right)}}{k} \cdot \Re\{c_{\gamma}(kn)\} \\
		&= 3\sum_{n,k \not\equiv 0 (p)} \frac{q^{\left(\frac{nk}{p}\right)}}{k} \cdot c(kn) - 2(p+1) \sum_{n \not\equiv 0(p)} \sum_{k \equiv 0(p)} \frac{q^{\left(\frac{nk}{p}\right)}}{k} \\
		&= 3\sum_{n,k \not\equiv 0 (p)} \frac{q^{\left(\frac{nk}{p}\right)}}{k} \cdot c(kn) - \frac{2(p+1)}{p} \sum_{n \not\equiv 0(p)} \sum_{k=1}^\infty \frac{q^{nk}}{k}.
	\end{align*}
	On the other hand, similarly to the proof of Proposition \ref{logR} we have
	\begin{align*}
		\log R(\tau) = 6\sum_{n \not\equiv 0(p)} \sum_{k=1}^\infty \frac{q^{kn}}{k} - 6\sum_{n \not\equiv 0 (p)} \sum_{k=1}^\infty \frac{q^{\left(\frac{nk}{p}\right)}}{k} \cdot c(kn).
	\end{align*}
	By isolating the terms containing $c(0)$, we obtain
	\begin{align*}
		\log R(\tau) &= 6\sum_{n \not\equiv 0(p)} \sum_{k=1}^\infty \frac{q^{kn}}{k} - 6\sum_{n,k \not\equiv 0 (p)} \frac{q^{\left(\frac{nk}{p}\right)}}{k} \cdot c(kn) - \frac{2(p-2)}{p} \sum_{n \not\equiv 0 (p)} \sum_{k=1}^\infty \frac{q^{nk}}{k} \\
		&= - 6\sum_{n,k \not\equiv 0 (p)} \frac{q^{\left(\frac{nk}{p}\right)}}{k} \cdot c(kn) + \frac{4(p+1)}{p} \sum_{n \not\equiv 0 (p)} \sum_{k=1}^\infty \frac{q^{nk}}{k},
	\end{align*}
	concluding the proof.
\end{proof}

We are now ready to prove Proposition \ref{logq<p}.

\begin{proof}[Proof of Proposition \ref{logq<p}]
By Remark \ref{rmk: qinfunddom} it suffices to prove the statement for $\tau \in \mathcal{F} + n$, where $n \in \Z$ and $\mathcal{F}$ is the standard fundamental domain, and such that $\tau$ does not lie on the lower boundary. Suppose first that $P$ is close to a cusp lying in the Galois orbit corresponding to $\mathcal{O}_{\cubes}$ (i.e., the case in which $\gamma=\operatorname{Id}$). Evaluating $U$ in $\tau = \tau_P$ we obtain $\operatorname{Ord}_q(U) \log |q| = -\log |R(\tau)| - 3 \log p + \log |U(\tau)|$, and the triangle inequality yields $|\operatorname{Ord}_q(U) \log |q|| \leq |\log |R(\tau)|| + |- 3 \log p + \log |U(\tau)||$. By Corollary \ref{Uint} we have $0 \leq \log |U(\tau)| \leq 3 \log p$, hence $|\log |U(\tau)| - 3 \log(p)| \le 3\log p$. Combining this with Proposition \ref{logR} and Lemma \ref{order} we finally obtain
$$\frac{p^2-1}{6p}|\log|q|| \le 3 \log p + \frac{4\pi^2 p\sqrt{p}}{3|\log|q||}.$$

Suppose instead that $P$ is close to a cusp lying in the other Galois orbit (i.e., the case in which $\gamma \ne \operatorname{Id}$). Evaluating $U \circ \gamma$ in $\tau=\gamma^{-1}\tau_P$ (lying in $\mathcal{F}$ or in $\mathcal{F} + \frac{p-1}{2}$ depending on the sign of $q$, as above) and proceeding in the same way (using Lemma \ref{ordergamma} and Proposition \ref{logRgamma} instead of Lemma \ref{order} and Proposition \ref{logR} respectively), we obtain the inequality
$$\frac{p^2-1}{12p}|\log|q|| \le 3 \log p + \frac{2\pi^2 p \sqrt{p}}{3|\log|q||}.$$
We now set $x=|\log|q||$. So far we have obtained, respectively for $\gamma=\operatorname{Id}$ and $\gamma \ne \operatorname{Id}$,
\begin{gather}
	\frac{p^2-1}{6p}x^2 - 3 x \log p - \frac{4\pi^2 p\sqrt{p}}{3} \le 0 \label{ineq1} \\
	\frac{p^2-1}{12p}x^2 - 3x \log p - \frac{2\pi^2 p\sqrt{p}}{3} \le 0. \label{ineq2}
\end{gather}
The first inequality implies the second, so (independently of whether $\gamma = \operatorname{Id}$ or not) we get that $x$ satisfies \eqref{ineq2}, and therefore
\begin{gather*}
	x \le \frac{18p \log p}{p^2-1} + \sqrt{\frac{18^2 p^2 (\log p)^2}{(p^2-1)^2} + \frac{8\pi^2 p^2 \sqrt{p}}{p^2-1}}.
\end{gather*}
Using that $\sqrt{a^2+b^2} \le a+b$ for all non-negative real numbers $a, b$, we obtain
\begin{align*}
	x &\le f(p) := \frac{36 p\log p}{p^2-1} + \frac{2\sqrt{2} \, \pi p \sqrt[4]{p}}{\sqrt{p^2-1}}.
\end{align*}
Since $f(t) < 30$ for all $t \in [2,100]$, we can assume $p > 100$, in which case we get
\begin{align*}
	x &\le f(p) \le \, \frac{36 \cdot 101 \cdot \log101}{100 \cdot 102} + \frac{2\sqrt{2} \, \pi \cdot 101}{10 \sqrt{102}} \cdot \sqrt[4]{p} \ \le \ \frac{2\sqrt{2} \, \pi \cdot 101}{10 \sqrt{102}} \cdot \sqrt[4]{p} + 1.65,
\end{align*}
which concludes the proof.\end{proof}

\begin{remark}\label{rmk: wlog p greater 100}
The proof of Proposition \ref{logq<p} shows
that if $E$ is a non-CM elliptic curve and $5 < p < 100$ is a prime such that
$\operatorname{Im} \rho_{E,p}$ is conjugate to $G(p)$, then $x=|\log |q|| \leq f(p) < 30$. In particular, this implies that every time that we assume $|\log |q|| \ge 30$ we may also assume $p>100$.
\end{remark}

\begin{corollary}\label{p<\boundonprimes}
	Let $\faktor{E}{\Q}$ be an elliptic curve without complex multiplication and set $q=e^{2\pi i \tau}$, where $\tau \in \uhp$ corresponds to the complex elliptic curve $E(\mathbb{C})$. Let $p>5$ be a prime number such that $\operatorname{Im} \rho_{E,p}$ is conjugate to $G(p)$. If $|\log|q|| \ge 30$, then $p<\boundonprimes$.
\end{corollary}
\begin{proof}
    By Theorem \ref{zywina} we can assume that $p \equiv 2 \pmod 3$. The statement then follows from Theorem \ref{effiso} (2), bounding $\Fheight(E)$ by Theorem \ref{heights-ineq} (2) and $|\log|q||$ by Proposition \ref{logq<p}. Note that the quantity $\Im\{\tau\}$, which appears in Theorem \ref{effiso} (2), is equal to $\frac{|\log |q||}{2\pi}$. We obtain an explicit inequality in $p$, which can be numerically solved.
\end{proof}

The upper bound on $p$ given by Corollary \ref{p<\boundonprimes} is sharper than the corresponding bound in Theorem \ref{lefournlemos}, but it is still not good enough to test all the remaining primes by the direct computation we describe in Section \ref{sec:Conclusion}. For this reason, in the next section we improve on Proposition \ref{logq<p}. The intermediate result we just obtained will be an important ingredient in this improvement.

\subsection{Abel summation and a sharper bound on $\log|q|$}

Proposition \ref{logq<p} improves the bound on $\log R$ given in \cite{lefournlemos21} by considering cancellation among roots of unity. In particular, the argument in \cite{lefournlemos21} used the trivial estimate $|c(k)| \le \frac{p-2}{3}$, which we replaced by $|c(k)| \le \frac{4}{3}\sqrt{p}$ using Lemma \ref{kloosterman}. In this section, we show that by rearranging the sums in $\log R$ using partial summation, we obtain an expression for $\log R$ which gives even more cancellations. This leads to a further improvement of the bound of Proposition \ref{logR} and ultimately to the following result, which supersedes Proposition \ref{logq<p}.

\begin{proposition}\label{logq<30}
	Let $\faktor{E}{\Q}$ be an elliptic curve without complex multiplication and set $q=e^{2\pi i \tau}$, where $\tau \in \uhp$ corresponds to the complex elliptic curve $E(\mathbb{C})$. If $p>5$ is a prime number such that $\operatorname{Im} \rho_{E,p}$ is conjugate to $G(p)$, then $|\log|q|| < 39$.
\end{proposition}

To prove this result, we can and do assume that $|\log|q|| \ge 30$ and hence (by Remark \ref{rmk: wlog p greater 100}) that $p > 100$. By Corollary \ref{p<\boundonprimes} and Theorem \ref{-1mod9} we can also assume that $p$ is less than $\boundonprimes$ and that it satisfies $p \equiv 2,5 \pmod 9$.
We keep the notation from the previous section. Similarly to the proof of Proposition \ref{logRgamma} we have
\begin{align*}
	\log R = - 6\sum_{n,k \not\equiv 0 (p)} \frac{q^{\frac{kn}{p}}}{k} \cdot c(kn) + \frac{4(p+1)}{p} \sum_{n \not\equiv 0 (p)} \sum_{k=1}^\infty \frac{q^{nk}}{k}.
\end{align*}
Writing $m=kn$ we have
\begin{equation}\label{logR=S+cosa}
	\log R = - 6\sum_{m \not\equiv 0 (p)} \sum_{k|m} \frac{q^{\frac{m}{p}}}{k} \cdot c(m) + \frac{4(p+1)}{p}\sum_{n \not\equiv 0(p)} \sum_{k=1}^\infty \frac{q^{kn}}{k}.
\end{equation}
Using Lemma \ref{logsum}, we bound the second term as follows:
\begin{equation}\label{eq: bound for the small subsum in log R}
\left|\frac{4(p+1)}{p}\sum_{n \not\equiv 0(p)} \sum_{k=1}^\infty \frac{q^{kn}}{k}\right| < \frac{4(p+1)}{p}\sum_{n=1}^\infty \sum_{k=1}^\infty \frac{|q|^{kn}}{k} < \frac{4(p+1)\pi^2}{6p|\log|q||}.
\end{equation}
This quantity is bounded uniformly in $p$. We now focus on the first sum in equation \eqref{logR=S+cosa}, which we denote by $S$. By partial summation, we have
\begin{align}\label{S-Abel}
	\begin{split}
		S &= - 6\sum_{m \not\equiv 0 (p)} \sum_{k|m} \frac{q^{\frac{m}{p}}}{k} \cdot c(m) = - 6\sum_{m \not\equiv 0 (p)} q^{\frac{m}{p}} \sum_{k|m} \frac{c(m)}{k} \\
		&= - 6\sum_{s=1}^\infty (q^\frac{s}{p}-q^\frac{s+1}{p}) D(s),
	\end{split}
\end{align}
where
\begin{equation}\label{eqD}
	D(s):= \sum_{\substack{m \le s \\ m \not\equiv 0(p)}} \sum_{k|m} \frac{c(m)}{k}.
\end{equation}
The idea of using this rewriting of $S$ is that, when $p$ is large, the factor $q^{\frac{s}{p}}-q^{\frac{s+1}{p}}=q^{\frac{s}{p}}(1-q^{\frac{1}{p}})$ becomes small, because $|q^{\frac{1}{p}} -1| = |e^{\frac{\log q}{p}}-1| \approx \frac{|\log |q||}{p}$. Provided that $D(s)$ does not grow too quickly, the factor of $p$ in the denominator leads to a much better upper bound on $S$, hence on $\log R$, than that provided by Proposition \ref{logR}. We now give two different estimates for $|D(s)|$, one for $s<p$ and one for $s \ge p$, in Lemmas \ref{bound D for s less than p} and \ref{bound D for s>p} respectively.

Consider all the primes $p$ smaller than a fixed bound $M$. For $s<p$ we have
\begin{equation*}
	D(s)=\sum_{m\le s} \sum_{k|m} \frac{c(m)}{k},
\end{equation*}
and we can write $|D(s)| \le C\sqrt{p}\sqrt{s}$ for some $C=C(M)$.

\begin{lemma}\label{bound D for s less than p}
	Let $M=\boundonprimes$ and $C=4.25$. We have $|D(s)| \le C\sqrt{p}\sqrt{s}$ for all $s<p<M$ with $p$ prime, $p \equiv 2, 5 \pmod{9}$.
\end{lemma}

\begin{proof}
	We get a suitable value of $C$ by explicitly computing the values of $D(s)$ for all primes $p \equiv 2,5 \pmod 9$ up to $M$ and for $s=1, \dots, p-1$. More precisely, in order to quickly compute $D(s)$ we obtain the values of the coefficients $c(m)$ using Rader's FFT algorithm \cite{rader68} applied to the characteristic function of the set $F(m)$. Indeed, every $c(m)$ is defined as the (non-normalised) Fourier transform of the characteristic function $\mathbbm{1}_{F(m)}$ of the set $F(m)$. Computing the fast Fourier transform is the most expensive step of the algorithm, taking time $O(p \log p)$. Since there are $O\left(\frac{M}{\log M}\right)$ primes up to $M$ (this remains true also restricting to the congruence classes $2, 5 \bmod 9$), the asymptotic complexity of the algorithm is $O(M^2)$. For $M=1.03 \cdot 10^5$, the run time of our implementation \cite{Script} is of a few hours on modest hardware.
\end{proof}

\begin{remark}
    It is important to notice that the value of $D(s)$ depends on the choice of $\varepsilon$ (see equation \eqref{eq:cartan}). All the calculations in this section, including in particular that of Lemma \ref{bound D for s less than p}, are performed by taking as $\varepsilon$ the image in $\mathbb{F}_p$ of the least positive integer which is a quadratic non-residue modulo $p$.
    \end{remark}

\begin{remark}
	For the computationally accessible values of $M$, one can check that the optimal $C(M)$ grows very slowly: for example, $C(10^4) \approx 3.789$, $C(10^5) \approx 4.246$ and $C(10^6) \approx 5.169$.
\end{remark}

\begin{remark}\label{rmk: heuristics for the bound s<p}
	The choice of the form of the bound in Lemma \ref{bound D for s less than p} is supported by the following heuristics. We assume that the coefficients $c(m)$ are pseudo-random values in the interval $\left[-\frac{4}{3}\sqrt{p}, \frac{4}{3}\sqrt{p}\right]$. Since $\sigma_{-1}(m) = \sum_{k|m} \frac{1}{k} = O(\log \log m)$, the quantity $D(s)$ is the sum of $s$ random values in the interval $\left[-\alpha\sqrt{p} \log\log s, \alpha\sqrt{p} \log\log s \right]$ for some constant $\alpha$, so we expect it to be $O(\sqrt{ps} (\log\log s)^2)$. By taking small values of $p$ (for example, $p<\boundonprimes$) and $s<p$, we can essentially treat $(\log\log s)^2$ as a constant.
\end{remark}

In the regime $s \ge p$, it will be enough to use the following easier upper bound on $D(s)$:
\begin{lemma}\label{bound D for s>p}
	Let $p$ be a prime and let $s\ge p$ be an integer. We have $|D(s)| < \frac{2\pi^2}{9}s\sqrt{p}$.
\end{lemma}
\begin{proof}
	It suffices to note that by Lemma \ref{kloosterman} we have
	\begin{align*}
		|D(s)| &\le \sum_{\substack{m\le s \\ m \not\equiv 0(p)}} \sum_{k|m} \frac{|c(m)|}{k} \le \frac{4}{3}\sqrt{p} \sum_{k=1}^s \sum_{\ell=1}^{\lfloor \frac{s}{k} \rfloor} \frac{1}{k} \le \frac{4}{3}\sqrt{p} \sum_{k=1}^s \frac{s}{k^2} < \frac{4\pi^2}{18}s\sqrt{p}.
	\end{align*} \qedhere
\end{proof}

We now combine these results to prove the following.
\begin{proposition}\label{AbellogR-general}
	Let $E, p, q$ be as in Proposition \ref{logq<p} and let $R$ be as in Definition \ref{def: Ra1a2 and R}. Let $D(s)$ be as in equation \eqref{eqD} and let $C$ be the minimum constant such that $|D(s)| \le C\sqrt{ps}$ for $s<p$. Writing $x:=|\log|q|| \ge 30$, we have
	\begin{align*}
		|\log|R|| &\le \frac{3Cp\sqrt{\pi}}{x^{\frac{1}{2}}} \cdot 1.28 + \frac{5}{3}\pi^2e^{-x}p\sqrt{p} + \frac{2(p+1)\pi^2}{3px}.
	\end{align*}
\end{proposition}
\begin{proof}
Since $2\pi \Im\tau = |\log|q|| \ge 30$, we know that $\Im\tau>1$, and hence we may also assume that $\tau$ is in the standard fundamental domain $\mathcal{F}$ for the action of $\operatorname{SL}_2(\Z)$, since every fundamental domain containing such a $\tau$ is obtained as $\mathcal{F}+n$ for $n \in \Z$, but integer translations do not change the value of $q$.
We start by estimating the sum $S$ defined in equation \eqref{S-Abel}, dividing it into two parts. Using Lemmas \ref{bound D for s less than p} and \ref{bound D for s>p} we obtain
\begin{align}\label{|S|-Abel}
	\begin{split}
		|S| &\le 6 \, |1-q^\frac{1}{p}| \, \sum_{s=1}^\infty |q|^\frac{s}{p} |D(s)| \\
		&\le 6C\sqrt{p} \, |1-q^\frac{1}{p}| \, \sum_{s=1}^{p-1} |q|^\frac{s}{p}\sqrt{s} \ + \ \frac{4}{3}\pi^2\sqrt{p} \, |1-q^\frac{1}{p}| \, \sum_{s=p}^\infty |q|^\frac{s}{p}s.
	\end{split}
\end{align}
We now use the following elementary fact: if $f:\mathbb{R}_{\ge 0} \to \mathbb{R}_{\ge 0}$ is a differentiable function with a single local maximum in $x_0 \in \mathbb{R}_{\ge 0}$, then $$\sum_{n=1}^\infty f(n) < \int_1^\infty f(x) dx + f(x_0).$$
Since $\frac{d}{ds}(|q|^\frac{s}{p}\sqrt{s}) = |q|^\frac{s}{p}\sqrt{s} \left(\log|q|^\frac{1}{p} + \frac{1}{2s}\right)$, the function $|q|^\frac{s}{p}\sqrt{s}$ is increasing for $s \le -\frac{1}{2\log|q|^{1/p}} = \frac{p}{2|\log|q||}$ and decreasing for larger values of $s$. We then have the following estimate:
\begin{align*}
	\sum_{s=1}^{p-1} |q|^\frac{s}{p}\sqrt{s} &< \sum_{s=1}^\infty |q|^\frac{s}{p}\sqrt{s} < \int_1^\infty |q|^\frac{s}{p}\sqrt{s} \, ds + |q|^\frac{1}{2|\log|q||}\sqrt{\frac{p}{2|\log|q||}} \\
	&< \int_0^\infty e^{-\frac{|\log|q||}{p}s}\sqrt{s} \, ds + e^{-\frac{1}{2}}\sqrt{\frac{p}{2|\log|q||}} \\
&= \frac{p\sqrt{p}}{|\log|q||^{\frac{3}{2}}}\frac{\sqrt{\pi}}{2} \left( 1 + \frac{\sqrt{2}|\log|q||}{p\sqrt{\pi e}} \right),
\end{align*}
where in the last equality we have used the classical Gaussian integral $\int_{0}^\infty e^{-t^2} \, dt = \frac{1}{2}\sqrt{\pi}$.
Using Proposition \ref{logq<p}, together with $p>100$ (see Remark \ref{rmk: wlog p greater 100}), we have
$$\frac{|\log|q||}{p} < \frac{2\sqrt{2} \, \pi \cdot 101}{10 \sqrt{102} \, p^{\frac{3}{4}}} + \frac{1.65}{p} \le \frac{2\sqrt{2} \, \pi \cdot \sqrt[4]{101}}{10 \sqrt{102}} + \frac{1.65}{101} < 0.3,$$
and so
$$\sum_{s=1}^{p-1} |q|^\frac{s}{p}\sqrt{s} < \frac{p\sqrt{p}}{|\log|q||^{\frac{3}{2}}}\frac{\sqrt{\pi}}{2} \cdot 1.15.$$
To estimate the sum of the terms with $s \ge p$, we use the following fact.
If $x \in (0,1)$, then
$$\sum_{s=p}^\infty x^ss = x \frac{d}{dx} \sum_{s=p}^\infty x^s = x \frac{d}{dx} \frac{x^p}{1-x} = x^p \left( \frac{p}{1-x} + \frac{x}{(1-x)^2} \right).$$
Putting everything together, equation \eqref{|S|-Abel} yields
\begin{equation*}
		|S| \le 3Cp\sqrt{p} \, |1-q^\frac{1}{p}| \frac{\sqrt{\pi p}}{|\log|q||^{\frac{3}{2}}} \cdot 1.15 + \frac{4}{3}\pi^2|q|\sqrt{p} \, |1-q^\frac{1}{p}| \left( \frac{p}{1-|q|^\frac{1}{p}} + \frac{|q|^\frac{1}{p}}{(1-|q|^\frac{1}{p})^2} \right).
\end{equation*}
Applying Lemmas \ref{lemminotecnico1} and \ref{lemminotecnico2} we obtain
\begin{align*}
	|S| &\le 3Cp\sqrt{p} \left(\frac{|\log|q||}{p} + \frac{\pi}{p}\right) \frac{\sqrt{\pi p}}{|\log|q||^{\frac{3}{2}}} \cdot 1.15 \\
	& \quad + \frac{4}{3}\pi^2|q|\sqrt{p} \left(1-|q|^\frac{1}{p} + |q|^\frac{1}{2p} \frac{\pi}{p}\right) \left( \frac{p}{1-|q|^\frac{1}{p}} + \frac{|q|^\frac{1}{p}}{(1-|q|^\frac{1}{p})^2} \right),
\end{align*}
which we rewrite as
\begin{align*}
	|S| &\le \frac{3Cp\sqrt{\pi}}{|\log|q||^{\frac{1}{2}}} \cdot 1.15 \left(1 + \frac{\pi}{|\log|q||} \right) \\
	& \quad + \frac{4}{3}\pi^2|q|\sqrt{p} \left(1 + \frac{|q|^\frac{1}{2p}}{1-|q|^\frac{1}{p}} \cdot \frac{\pi}{p}\right) \left( p + \frac{|q|^\frac{1}{p}}{1-|q|^\frac{1}{p}} \right).
\end{align*}
Using Lemma \ref{lemminotecnico2} again we have
$$\frac{|q|^\frac{1}{2p}}{1-|q|^\frac{1}{p}} = \frac{|q|^\frac{1}{2p}}{(1-|q|^\frac{1}{2p})(1+|q|^\frac{1}{2p})} < \frac{1}{1+|q|^\frac{1}{2p}} \cdot \frac{2p}{|\log|q||} < \frac{2p}{|\log|q||},$$
and using the assumption $|\log|q|| \ge 30$ we obtain
\begin{align*}
	|S| &< \frac{3Cp\sqrt{\pi}}{|\log|q||^{\frac{1}{2}}} \cdot 1.28 + \frac{4}{3}\pi^2|q|\sqrt{p} \left(1 + \frac{2\pi}{|\log|q||} \right) \left( p + \frac{p}{|\log|q||} \right) \\
	&\le \frac{3Cp\sqrt{\pi}}{|\log|q||^{\frac{1}{2}}} \cdot 1.28 + \frac{4}{3}\pi^2|q|p\sqrt{p} \cdot 1.25.
\end{align*}
    Bounding the sums in equation \eqref{logR=S+cosa} and recalling equation \eqref{eq: bound for the small subsum in log R}, we have obtained
	$$|\log|R|| \le \frac{3Cp\sqrt{\pi}}{x^{\frac{1}{2}}} \cdot 1.28 + \frac{5}{3}\pi^2|q|p\sqrt{p} + \frac{2(p+1)\pi^2}{3px},$$
	which concludes the proof.
\end{proof}

\begin{corollary}\label{AbellogR}
	Let $E, p, q$ be as in Corollary \ref{p<\boundonprimes} and let $R$ be as in Definition \ref{def: Ra1a2 and R}. If $|\log|q|| \ge 30$, then
	\begin{align*}
		|\log|R|| &\le 29 \cdot \frac{p}{|\log|q||^{\frac{1}{2}}} + 0.23.
	\end{align*}
\end{corollary}
\begin{proof}
    By Proposition \ref{canonicalsgr} we have $p^4 \mid j(E)$, and Theorem \ref{estimate-qj} implies first $30 \le x \le \log(|j(E)|+970.8)$, so $|j(E)|>3500$, and then $|j(E)| \le \frac{2}{|q|}$. Therefore, we have $p^4|q| \le |j(E)|\cdot |q| \le 2$. By Proposition \ref{AbellogR-general} we obtain
    \begin{align*}
		|\log|R|| &\le \frac{3Cp\sqrt{\pi}}{x^{\frac{1}{2}}} \cdot 1.28 + \frac{10\pi^2}{3p^2\sqrt{p}} + \frac{2(p+1)\pi^2}{3px}.
	\end{align*}
    The result follows by using $C \le 4.25$ (Lemma \ref{bound D for s less than p}, which we can use thanks to Corollary \ref{p<\boundonprimes}), $p \ge 101$ (Remark \ref{rmk: wlog p greater 100}) and $x \ge 30$.
\end{proof}

We now notice that all the arguments we applied to $\log R$ are also valid for $\log R_\gamma$. We can then give an analogue of Corollary \ref{AbellogR} for $\log R_\gamma$.
As before, we have
\begin{align*}
    \log R_\gamma &= -6\sum_{n,k \not\equiv 0 (p)} \frac{q^{\frac{nk}{p}}}{k} \cdot c_\gamma(kn) - \frac{2(p+1)}{p} \sum_{n \not\equiv 0(p)} \sum_{k=1}^\infty \frac{q^{nk}}{k} \\
    &= - 6\sum_{s=1}^\infty (q^\frac{s}{p}-q^\frac{s+1}{p}) D_\gamma(s) - \frac{2(p+1)}{p} \sum_{n \not\equiv 0(p)} \sum_{k=1}^\infty \frac{q^{nk}}{k},
\end{align*}
where $$D_\gamma(s):= \sum_{\substack{m \le s \\ m \not\equiv 0(p)}} \sum_{k|m} \frac{c_\gamma(m)}{k}$$
is an analogue of $D(s)$ in this context.

\begin{lemma}\label{lemma: value of C_gamma}
    For every prime $p<\boundonprimes$ with $p \equiv 2, 5 \pmod{9}$, the following hold:
    \begin{enumerate}
        \item for $s<p$, we have $|D_\gamma(s)| < C_\gamma \sqrt{ps}$ with $C_\gamma = 2.81$.
        \item for $s \ge p$, we have $|D_\gamma(s)| < \frac{2\pi^2}{9} s\sqrt{p}.$
    \end{enumerate}
\end{lemma}
\begin{proof}
    The proof is analogous to those of Lemmas \ref{bound D for s less than p} and \ref{bound D for s>p}.
\end{proof}

Reasoning as in the proof of Proposition \ref{AbellogR-general} we obtain the following.
\begin{proposition}
    Let $E, p, q$ be as in Proposition \ref{logq<p} and let $R_\gamma$ be as in equation \eqref{eq: definition of R_gamma}. Let $C_\gamma$ be the minimum constant such that $|D_\gamma(s)| \le C_\gamma\sqrt{ps}$ for $s<p$. Writing $x:=|\log|q|| \ge 30$, we have
	\begin{align*}
		|\log|R_\gamma|| &\le \frac{3C_\gamma p\sqrt{\pi}}{x^{\frac{1}{2}}} \cdot 1.28 + \frac{5}{3}\pi^2e^{-x}p\sqrt{p} + \frac{(p+1)\pi^2}{3px}.
	\end{align*}
\end{proposition}
Since we know that it suffices to consider primes up to $\boundonprimes$ (Corollary \ref{p<\boundonprimes}), we can use the bound on the value of $C_\gamma$ provided by Lemma \ref{lemma: value of C_gamma} to obtain the following numerical estimate.
\begin{corollary}\label{AbellogR-gamma}
    Let $E, p, q$ be as in Corollary \ref{p<\boundonprimes} and let $R_\gamma$ be as in equation \eqref{eq: definition of R_gamma}. If $|\log|q|| \ge 30$, then
	\begin{align*}
		|\log|R_\gamma|| &\le 19.13 \cdot \frac{p}{|\log|q||^{\frac{1}{2}}} + 0.12.
	\end{align*}
\end{corollary}

We can finally prove the main result of this section.
\begin{proof}[Proof of Proposition \ref{logq<30}]
	Suppose $|\log|q|| \ge 30$. As in the previous section, from the two possible equations
	\begin{gather*}
		\log|U|= \operatorname{Ord}_q(U) \log |q|+ 3 \log p + \log |R| \\
		\log|U \circ \gamma| = \operatorname{Ord}_q(U \circ \gamma) \log|q| + \log |R_{\gamma}|
	\end{gather*}
	we obtain the inequalities
	\begin{gather*}
	    \frac{p^2-1}{6p}|\log|q|| \le 3 \log p + |\log|R||, \\
	    \frac{p^2-1}{12p}|\log|q|| \le 3 \log p + |\log|R_\gamma||.
	\end{gather*}
	Comparing Corollary \ref{AbellogR} with Corollary \ref{AbellogR-gamma} we notice that it suffices to consider the second inequality. Writing $x=|\log|q||$ we have
	$$\frac{p^2-1}{12p}x \le 3 \log p + 19.13 \frac{p}{\sqrt{x}} + 0.12.$$
	By Remark \ref{rmk: wlog p greater 100} we may assume $p>100$, 
    hence
	$$x\sqrt{x} - 2\sqrt{x} -230 \le 0,$$
	which implies $x < 39$.
\end{proof}

\section{Conclusion of the proof of Theorems \ref{solution} and \ref{solution for small primes}}\label{sec:Conclusion}

We recall the main content of Theorem \ref{solution for small primes} (which implies Theorem \ref{solution}): there exists no pair $(E, p)$, where $E$ is an elliptic curve over $\Q$ without CM and $p>5$ is a prime, for which the image of the representation $\rho_{E,p}$ is the group $G(p)$ (up to conjugacy). 
Suppose by contradiction that such a pair exists. We consider the base change of $E$ to $\C$ (along the unique embedding $\mathbb{Q} \hookrightarrow \mathbb{C}$). There is a unique $\tau$ in the standard fundamental domain $\mathcal{F}$ of the upper half plane $\uhp$ that corresponds to $E(\C)$; we set $q=e^{2\pi i \tau}$.
In this setting, in the previous sections we have proved the following properties:
\begin{itemize}
	\item \underline{$j(E)$ is an integer}: follows from Lemma \ref{lemma:theREALintejer};
	\item \underline{$p \equiv 2,5 \pmod 9$}: follows from Theorem \ref{-1mod9};
	\item \underline{$p^4 \mid j(E)$}: follows from Proposition \ref{canonicalsgr};
	\item \underline{$|j(E)| \le 2\cdot e^{39}$}: follows from Proposition \ref{logq<30} and Theorem \ref{estimate-qj};
	\item \underline{$19 \leq p<20400$}: the first inequality follows from Lemma \ref{lemma:theREALintejer}. For the second, we know that $p^4 \le |j(E)| \le 2 \cdot e^{39}$, hence $p \le \sqrt[4]{2} \cdot e^{\frac{39}{4}} < 20400$;
	\item \underline{$j(E)=p^d \cdot c^3$ for $d \in \{4,5\}$ and $c \in \Z$}: by Lemma \ref{lem:jisalmostacube}, we know that $j(E)=p^d \cdot c^3$, and by Proposition \ref{canonicalsgr} we also know that $d \ge 4$. We can assume that $d \in \{4,5,6\}$, since higher exponents can be reduced modulo 3 by reabsorbing the factors of $p$ in $c^3$. Moreover, by Lemma \ref{p|j} the case $d=6$ does not occur, hence we can assume that $d \in \{4,5\}$.
\end{itemize}

To complete the proof of Theorem \ref{solution} and Theorem \ref{solution for small primes}, we check directly, for all primes $19 \leq p<20400$ in the relevant congruence classes modulo $9$, whether there exists any pair $\left(\faktor{E}{\mathbb{Q}}, p\right)$ as above. To be able to test a finite number of curves, we also need the following well-known lemma.
\begin{lemma}\label{lemma: twist invariance}
	If $E$ and $E'$ are two non-CM elliptic curves over $\Q$ with $j(E)=j(E')$, $p>2$ is a prime, and $H \subseteq \operatorname{GL}(E[p])$ is a subgroup that contains $-\operatorname{Id}$, then $\operatorname{Im}\rho_{E,p} \subseteq H$ if and only if $\operatorname{Im}\rho_{E',p} \subseteq H$.
\end{lemma}
\begin{proof}
	Since either $E' \cong E$ or $E'$ is a quadratic twist of $E$, the statement follows from \cite[Corollary 5.25]{sutherland16}.
\end{proof}

We now proceed as follows (see \cite{Script}):
\begin{enumerate}
	\item For every prime $p \equiv 2,5 \pmod 9$ with $19 \leq p<20400$, every $d \in \{4,5\}$ and every integer $c \neq 0$ in the interval $\left[-\sqrt[3]{2} \cdot e^{13} p^{-\frac{d}{3}}, \sqrt[3]{2} \cdot e^{13} p^{-\frac{d}{3}} \right]$, we take an integral model $\mathcal{E}$ of a curve $E/\Q$ with $j$-invariant $j(E)=p^d \cdot c^3$.
	\item We loop over primes $\ell$ distinct from $p$, in increasing order. For each such prime $\ell$:
	\begin{enumerate}
		\item We check if $\mathcal{E}$ has good reduction at $\ell$. If it does, we continue with (b); otherwise, we move on to the next prime $\ell$.
		\item We compute $a_\ell= \ell+1-|\widetilde{\mathcal{E}}(\F_\ell)|$ by counting the $\mathbb{F}_\ell$-rational points of $\mathcal{E}$ modulo $\ell$.
		\item We check whether the roots of the polynomial $t^2-a_\ell t + \ell \in \F_p[t]$ are cubes in $\F_{p^2}^\times$ (note that $\lambda \in \F_{p^2}^\times$ is a cube if and only if $\lambda^{\frac{p^2-1}{3}}=1$). If they are cubes, we continue with the next prime $\ell$. If they are not, we mark $j(E) = p^d \cdot c^3$ as \textit{ruled out} and continue with the next candidate $(p, d, c)$.
	\end{enumerate}
\end{enumerate}
The algorithm terminates, in the sense that every candidate $j$-invariant is marked as \textit{ruled out}: the loop in step 2 is always broken by finding some prime $\ell$ (in fact, $\ell < 200$ in all cases) for which the roots of $t^2-a_\ell t+ \ell$ are not cubes in $\F_{p^2}^\times$. We claim that this proves Theorem \ref{solution}, as well as Theorem \ref{solution for small primes} for $p>5$.
Indeed, suppose by contradiction that there exists a pair $(E_1, p)$ such that $E_1$ is a non-CM elliptic curve over $\Q$ and $p>5$ is a prime for which $\operatorname{Im} \rho_{E_1, p}$ is conjugate to $G(p)$. Then, by the discussion above we know that $j(E_1)$ is of the form $p^d \cdot c^3$ for some $p, d, c$ satisfying the conditions in step $1$ (note that $j=0$ gives a CM elliptic curve), so the curve $E$ we construct in this step is a quadratic twist of $E_1$. By Lemma \ref{lemma: twist invariance}, the image of $\rho_{E, p}$ is conjugate to a subgroup of $G(p)$ (note that $-\operatorname{Id} \in G(p)$), and by fixing a basis, we can assume that it is in fact contained in $G(p)$.

On the other hand, let $\ell$ be a prime for which the roots of $t^2-a_\ell t+ \ell$ are not cubes in $\F_{p^2}^\times$ (the output of the algorithm shows that such a prime exists), and let $F_\ell \in \Gal\left(\faktor{\overline{\Q}}{\Q}\right)$ be a Frobenius corresponding to $\ell$. The element $\rho_{E, p}(F_\ell)$ has characteristic polynomial $t^2-a_\ell t+ \ell$.
Since $\rho_{E, p}(F_\ell)$ is in $G(p)$, it satisfies at least one of the following: $a_\ell=0$ (if $\rho_{E, p}(F_\ell)$ lies in the normaliser $C_{ns}^+(p)$, but not in the Cartan $C_{ns}(p)$ itself), or $\rho_{E,p}(F_\ell)$ is the cube of some element $g_\ell$ in $\operatorname{GL}_2(\F_p)$ (if it lies in the subgroup of cubes of $C_{ns}(p)$). In both cases, the eigenvalues of $\rho_{E, p}(F_\ell)$ are cubes in $\mathbb{F}_{p^2}^\times$: if $a_\ell=0$, this follows from the fact that the roots of the characteristic polynomial are $\pm \sqrt{-\ell}$, and $-\ell$ is a cube in $\F_p^\times$ since $p \equiv 2 \pmod 3$; if $a_\ell \neq 0$, it follows from the fact that the eigenvalues of $\rho_{E, p}(F_\ell)$ are the cubes of the eigenvalues of $g_\ell$. However, the choice of $\ell$ shows that the eigenvalues of $\rho_{E, p}(F_\ell)$ are \textit{not} cubes in $\F_{p^2}^\times$: the contradiction shows that the pair $(E_1, p)$ cannot exist, which concludes the proof of Theorem \ref{solution}.

To finish the proof of Theorem \ref{solution for small primes} it suffices to notice that for $p=5$ there are many curves $E$ for which $\operatorname{Im}\rho_{E,5}$ is conjugate to $G(5)$, as suggested by \cite[Theorem 1.4 (ii)]{zywina15}: for example, we can take the curve $y^2 = x^3 - 950x - 11480$, with LMFDB label \href{https://www.lmfdb.org/EllipticCurve/Q/70400/bg/1}{70400.bg1}.

\begin{remark}
	Our algorithm \cite{Script} terminates in around 2 minutes. Since the running time is clearly exponential in the bound on $\log |j(E)|$, it would have been impossible to carry out this calculation without a sharp absolute bound on $\log|q|$, such as that given by Proposition \ref{logq<30}. To showcase the sharpness of our bound, we point out that even just knowing $\log |j(E)| < 50$ would have led to a perfectly tractable computation: our algorithms test all $j$ with $\log |j| < 50$ in about an hour and a half. On the other hand, it is clear that the bound $\log |j(E)| < 162$ which follows from Proposition \ref{logq<p} and Corollary \ref{p<\boundonprimes} would have been too loose to carry out the final calculation as described in this section. Indeed, knowing $\log |j(E)| < 39 + \log 2$ we have to test 105860 pairs ($j$-invariant, prime); using only $\log |j(E)| < 50$ the number rises to $\approx 4.2 \cdot 10^6$, and with $\log |j(E)| < 162$ to $\approx 6.8 \cdot 10^{22}$.
\end{remark}

\bibliographystyle{abbrv}
\bibliography{bib}
	
\end{document}